\documentclass[runningheads]{llncs}
\usepackage{latexsym, amssymb, amsfonts,amsmath}
\usepackage{graphicx}
\usepackage[all,knot,arc,import,poly]{xy}

\begin{document}
\title{Classifying toposes for some theories of $\mathcal{C}^{\infty}-$rings}
%
%
\author{Jean Cerqueira Berni\inst{1} \and
Hugo Luiz Mariano\inst{2}}
\authorrunning{J. C. Berni and H. L. Mariano}
%
\institute{Institute of Mathematics and Statistics, University of S\~{a}o Paulo,
Rua do Mat\~{a}o, 1010, S\~{a}o Paulo - SP, Brazil. \email{jeancb@ime.usp.br} \and
Institute of Mathematics and Statistics, University of S\~{a}o Paulo,
Rua do Mat\~{a}o, 1010, S\~{a}o Paulo - SP, Brazil, \email{hugomar@ime.usp.br}}
\maketitle              
%
%

\begin{abstract}
	In this paper we present classifying toposes for the following theories: the theory of $\mathcal{C}^{\infty}-$rings, the theory of local $\mathcal{C}^{\infty}-$rings and the theory of von Neumann regular $\mathcal{C}^{\infty}-$rings. The classifying toposes for the first two theories were stated without proof by Ieke Moerdijk and Gonzalo Reyes on the page 366 of \cite{MSIA}, where they assert  that the topos ${\rm \bf Set}^{\mathcal{C}^{\infty}{\rm \bf Rng}_{\rm fp}}$ classifies the theory of $\mathcal{C}^{\infty}-$rings and that the smooth Zariski topos classifies the theory of local $\mathcal{C}^{\infty}-$rings.
We also give a description of the classifying topos for the theory of von Neumann regular $\mathcal{C}^{\infty}-$rings.
\end{abstract}



\section*{Introduction}

\hspace{0.6cm} Loosely speaking, a $\mathcal{C}^{\infty}-$ring is an structure that interprets all symbols of (finitary) smooth real functions, preserving all the equational relations between them. According to I. Moerdijk and G. Reyes in \cite{rings1}, the original motivation to introduce and study $\mathcal{C}^{\infty}-$rings was to construct topos-models for Synthetic Differential Geometry. Their introduction circumvent some obstacles for a synthetic framing for Differential Geometry in ${\rm \bf Set}$, like, for instance, the lack, in the category of smooth manifolds, of finite inverse limits (in particular, even binary pullbacks of $\mathcal{C}^{\infty}-$manifolds are not manifolds, unless a condition of tranversality is fulfilled) and the absence of a convenient language to deal explicitly and directly with structures in the``infinitely small" level (cf. \cite{MSIA}). The existence of nilpotent elements, which provides us with a language that legitimates the use of geometric intuition does not come for free: the essential Kock-Lawvere axiom and its consequences, for example, are not compatible with the principle of the excluded middle (see \cite{Kock}). Thus, in order to deal with $\mathcal{C}^{\infty}-$rings one must give up on Classical Logic, and this necessarily leads us to the need for ``toposes'' - which can be seen as ``mathematical worlds'' that are governed by an internal intuitionistic logic.\\

The theory of $\mathcal{C}^{\infty}-$rings can be interpreted in any category $\mathcal{C}$ with finite products. However, as we consider theories of $\mathcal{C}^{\infty}-$rings that require its models to satisfy axioms with connectives such as ``disjunctions'' (which is the case for the theory of local $\mathcal{C}^{\infty}-$rings), we need ``richer categorical constructions'' (such as the possibility of forming unions of subobjects) in order to interpret them meaningfully  in any topos. \\

It is a well-known result that some types of first order theories - depending on the language and on the structure of their axioms  always have a classifying topos (cf. \cite{MakkaiReyes}). Among the first order theories which have a classifying topos we find the so-called ``geometric theories", \textit{i.e.}, theories (possibly infinitary and poli-sorted) whose axioms consist of implications between geometric formulas.\\

In this paper we are concerned with a concrete description of the classifying topoi of the (equational) theory of $\mathcal{C}^{\infty}-$rings, the (geometric) theory of the local $\mathcal{C}^{\infty}-$rings and the (equational) theory of von Neumann regular $\mathcal{C}^{\infty}-$rings. We present a step-by-step construction of such topoi, mimicking the construction of the classifying topoi for the theory of rings and for the theory of local rings given in \cite{ShvLog} with some adaptations. \\

\subsection*{Overview of the Paper}

The organization of this paper is as follows.\\

In the first section we present some concepts and preliminary results on categorial logic, classifying toposes and $\mathcal{C}^{\infty}-$rings.\\

In section 2 we give a comprehensive description of the classifying topos for the theory of $\mathcal{C}^{\infty}-$rings as a presheaf category. In the third section we give a detailed description of the smooth Zariski (Grothendieck) topology and its corresponding sheaf topos as the classifying topos for the theory of local $\mathcal{C}^{\infty}-$rings.\\

In the final section we introduce the notion of a von Neumann regular $\mathcal{C}^{\infty}-$ring along with some of its characterizations and we describe the classifying topos for the (first-order) theory of von Neumann regular $\mathcal{C}^{\infty}-$rings. We also present some related topics  which  can be developed in future works.

\section{Preliminaries}

\subsection{On $C^\infty$-rings}

\hspace{0.5cm}In order to formulate and study the concept of $\mathcal{C}^{\infty}-$ring, we are going to use a first order language $\mathcal{L}$ with a denumerable set of variables (${\rm \bf Var}(\mathcal{L}) = \{ x_1, x_2, \cdots, x_n, \cdots\}$) whose nonlogical symbols are the symbols of $\mathcal{C}^{\infty}-$functions from $\mathbb{R}^m$ to $\mathbb{R}^n$, with $m,n \in \mathbb{N}$, \textit{i.e.}, the non-logical symbols consist only of function symbols, described as follows:\\

For each $n \in \mathbb{N}$, the $n-$ary \textbf{function symbols} of the set $\mathcal{C}^{\infty}(\mathbb{R}^n, \mathbb{R})$, \textit{i.e.}, $\mathcal{F}_{(n)} = \{ f^{(n)} | f \in \mathcal{C}^{\infty}(\mathbb{R}^n, \mathbb{R})\}$. So the set of function symbols of our language is given by:
      $$\mathcal{F} = \bigcup_{n \in \mathbb{N}} \mathcal{F}_{(n)} = \bigcup_{n \in \mathbb{N}} \mathcal{C}^{\infty}(\mathbb{R}^n)$$
      Note that our set of constants is $\mathbb{R}$, since it can be identified with the set of all $0-$ary function symbols, \textit{i.e.}, ${\rm \bf Const}(\mathcal{L}) = \mathcal{F}_{(0)} = \mathcal{C}^{\infty}(\mathbb{R}^0) \cong \mathcal{C}^{\infty}(\{ *\}) \cong \mathbb{R}$.\\

The terms of this language are defined, in the usual way, as the smallest set which comprises the individual variables, constant symbols and $n-$ary function symbols followed by $n$ terms ($n \in \mathbb{N}$).\\

Apart from the functorial definition we gave in the introduction, we have many equivalent descriptions. We focus, first, in the following description of a $\mathcal{C}^{\infty}-$ring in ${\rm \bf Set}$.\\

\begin{definition}\label{cabala} A \textbf{$\mathcal{C}^{\infty}-$structure} on a set $A$ is a pair $ \mathfrak{A} =(A,\Phi)$, where:

$$\begin{array}{cccc}
\Phi: & \bigcup_{n \in \mathbb{N}} \mathcal{C}^{\infty}(\mathbb{R}^n, \mathbb{R})& \rightarrow & \bigcup_{n \in \mathbb{N}} {\rm Func}\,(A^n; A)\\
      & (f: \mathbb{R}^n \stackrel{\mathcal{C}^{\infty}}{\to} \mathbb{R}) & \mapsto & \Phi(f) := (f^{A}: A^n \to A)
\end{array},$$

that is, $\Phi$ interprets the \textbf{symbols}\footnote{here considered simply as syntactic symbols rather than functions.} of all smooth real functions of $n$ variables as $n-$ary function symbols on $A$.
\end{definition}

We call a $\mathcal{C}^{\infty}-$struture $\mathfrak{A} = (A, \Phi)$ a \textbf{$\mathcal{C}^{\infty}-$ring} if it preserves  projections and all equations between smooth functions. We have the following:

\begin{definition}\label{CravoeCanela}Let $\mathfrak{A}=(A,\Phi)$ be a $\mathcal{C}^{\infty}-$structure. We say that $\mathfrak{A}$ (or, when there is no danger of confusion, $A$) is a \textbf{$\mathcal{C}^{\infty}-$ring} if the following is true:\\

$\bullet$ Given any $n,k \in \mathbb{N}$ and any projection $p_k: \mathbb{R}^n \to \mathbb{R}$, we have:

$$\mathfrak{A} \models (\forall x_1)\cdots (\forall x_n)(p_k(x_1, \cdots, x_n)=x_k)$$

$\bullet$ For every $f, g_1, \cdots g_n \in \mathcal{C}^{\infty}(\mathbb{R}^m, \mathbb{R})$ with $m,n \in \mathbb{N}$, and every $h \in \mathcal{C}^{\infty}(\mathbb{R}^n, \mathbb{R})$ such that $f = h \circ (g_1, \cdots, g_n)$, one has:
$$\mathfrak{A} \models (\forall x_1)\cdots (\forall x_m)(f(x_1, \cdots, x_m)=h(g(x_1, \cdots, x_m), \cdots, g_n(x_1, \cdots, x_m)))$$
\end{definition}

\begin{definition}Let $(A, \Phi)$ and $(B,\Psi)$ be two $\mathcal{C}^{\infty}-$rings. A function $\varphi: A \to B$ is called a \textbf{morphism of $\mathcal{C}^{\infty}-$rings} or \textbf{$\mathcal{C}^{\infty}$-homomorphism} if for any $n \in \mathbb{N}$ and any $f: \mathbb{R}^n \stackrel{\mathcal{C}^{\infty}}{\to} \mathbb{R}$ the following diagram commutes:
$$\xymatrixcolsep{5pc}\xymatrix{
A^n \ar[d]_{\Phi(f)}\ar[r]^{\varphi^{(n)}} & B^n \ar[d]^{\Psi(f)}\\
A \ar[r]^{\varphi^{}} & B
}$$
 \textit{i.e.}, $\Psi(f) \circ \varphi^{(n)} = \varphi^{} \circ \Phi(f)$.
\end{definition}

\begin{remark}
Observe that $\mathcal{C}^{\infty}-$structures, together with their morphisms compose a category, that we denote by $\mathcal{C}^{\infty}{\rm \bf Str}$, and that $\mathcal{C}^{\infty}-$rings, together with all the $\mathcal{C}^{\infty}-$homomorphisms between $\mathcal{C}^{\infty}-$rings compose a full subcategory of $\mathcal{C}^{\infty}{\rm \bf Rng}$. In particular, since $\mathcal{C}^{\infty}{\rm \bf Rng}$ is a ``variety of algebras'' (it is a class of $\mathcal{C}^{\infty}-$structures which satisfy a given set of equations), it is closed under substructures, homomorphic images and producs, by \textbf{Birkhoff's HSP Theorem}. Moreover: 

$\bullet$ $\mathcal{C}^{\infty}{\rm \bf Rng}$ is a concrete category and the forgetful functor, $ U :  \mathcal{C}^{\infty}{\rm \bf Rng}  \to Set$ creates directed inductive colimits. Since $\mathcal{C}^{\infty}{\rm \bf Rng}$ is a variety of algebras, it has all (small) limits and (small) colimits. In particular, it has binary coproducts, that is, given any two $\mathcal{C}^{\infty}-$rings $A$ and $B$, we have their coproduct $A \stackrel{\iota_A}{\rightarrow} A\otimes_{\infty} \stackrel{\iota_B}{\leftarrow} B$;

$\bullet$ Each set $X$ freely generates a $C^\infty$-ring, $L(X)$, as follows:\\
-  for any finite set $X'$ with $\sharp X' = n$ we have $ L(X')= \mathcal{C}^{\infty}(\mathbb{R}^{X'}) \cong \mathcal{C}^\infty(\mathbb{R}^n, \mathbb{R})$ is the free $C^\infty$-ring on $n$ generators, $n \in \mathbb{N}$;\\
- for a general set, $X$, we take $L(X) = \mathcal{C}^{\infty}(\mathbb{R}^X):= \varinjlim_{X' \subseteq_{\rm fin} X} \mathcal{C}^{\infty}(\mathbb{R}^{X'})$;

$\bullet$ Given any $\mathcal{C}^{\infty}-$ring $A$ and a set, $X$, we can freely adjoin the set $X$ of variables to $A$ with the following construction: $A\{ X\}:= A \otimes_{\infty} L(X)$. The elements of $A\{ X\}$ are usually called $\mathcal{C}^{\infty}-$polynomials;

$\bullet$ The congruences of $\mathcal{C}^{\infty}-$rings are classified by their ``ring-theoretical'' ideals;

$\bullet$ Every $\mathcal{C}^{\infty}-$ring is the homomorphic image of some free $\mathcal{C}^{\infty}-$ring determined by some set, being isomorphic to the quotient of a free $\mathcal{C}^{\infty}-$ring by some ideal.\\

\end{remark}

Within the category of $\mathcal{C}^{\infty}-$rings, we have two special subcategories, that we define in the sequel.\\

\begin{definition}A $\mathcal{C}^{\infty}-$ring $A$ is \textbf{finitely generated} whenever there is some $n \in \mathbb{N}$ and some ideal $I \subseteq \mathcal{C}^{\infty}(\mathbb{R}^n)$ such that $A \cong \dfrac{\mathcal{C}^{\infty}(\mathbb{R}^{n})}{I}$. The category of all finitely generated $\mathcal{C}^{\infty}-$rings is denoted by $\mathcal{C}^{\infty}{\rm \bf Rng}_{\rm fg}$.
\end{definition}

\begin{definition}
A $\mathcal{C}^{\infty}-$ring is \textbf{finitely presented} whenever there is some $n \in \mathbb{N}$ and some \underline{finitely generated ideal} $I \subseteq \mathcal{C}^{\infty}(\mathbb{R}^n)$ such that $A \cong \dfrac{\mathcal{C}^{\infty}(\mathbb{R}^{n})}{I}$.\\

Whenever $A$ is a finitely presented $\mathcal{C}^{\infty}-$ring, there is some $n \in \mathbb{N}$ and some $f_1, \cdots, f_k \in \mathcal{C}^{\infty}(\mathbb{R}^n)$ such that:\\

$$ A = \dfrac{\mathcal{C}^{\infty}(\mathbb{R}^n)}{\langle f_1, \cdots, f_k\rangle}$$

The category of all finitely presented $\mathcal{C}^{\infty}-$rings is denoted by $\mathcal{C}^{\infty}{\rm \bf Rng}_{\rm fp}$
\end{definition}

\begin{remark}
The categories $\mathcal{C}^{\infty}{\rm \bf Rng}_{{\rm fg}}$ and $\mathcal{C}^{\infty}{\rm \bf Rng}_{{\rm fp}}$ are closed under initial objects, binary coproducts and binary coequalizers. Thus, they are finitely co-complete categories, that is, they have all finite colimits (for a proof of this fact we refer to the chapter 1 of \cite{tese}).\\

Since $\mathcal{C}^{\infty}{\rm \bf Rng}_{{\rm fp}}$ has all finite colimits, it follows that $\mathcal{C}^{\infty}{\rm \bf Rng}_{{\rm fp}}^{{\rm op}}$ has all finite limits.
\end{remark}

\begin{remark}
An $\mathbb{R}-$algebra $A$ in a category with finite limits, $\mathcal{C}$, may be regarded as a finite product preserving functor from the category ${\rm \bf Pol}$, whose objects are given by ${\rm Obj}\,({\rm \bf Pol}) =\{ \mathbb{R}^n | n \in \mathbb{N}\}$, and whose morphisms are given by polynomial functions between them, ${\rm Mor}\,({\rm \bf Pol}) = \{ \mathbb{R}^m \stackrel{p}{\rightarrow} \mathbb{R}^n | m,n \in \mathbb{N}, p \,\, {\rm polynomial}\}$, to $\mathcal{C}$, that is:

$$A: {\rm \bf Pol} \rightarrow \mathcal{C}.$$

In this sense, an $\mathbb{R}-$algebra $A$ is a functor which interprets all polynomial maps $p: \mathbb{R}^m \to \mathbb{R}^n$, for $m,n \in \mathbb{N}$. More precisely, the categories of $\mathbb{R}-$algebras as defined by the ``Universal Algebra approach'' and by the ``Functorial sense'' provide equivalent categories.

In this vein, one may define a $\mathcal{C}^{\infty}-$ring as a finite product preserving functor from the category ${\cal C}^{\infty}$, whose objects are given by ${\rm Obj}\,(\mathcal{C}^{\infty}) = \{ \mathbb{R}^n | n \in \mathbb{N}\}$ and whose morphisms are given by $\mathcal{C}^{\infty}-$functions between them, ${\rm Mor}\,(\mathcal{C}^{\infty}) = \{ \mathbb{R}^m \stackrel{f}{\rightarrow} \mathbb{R}^n | m,n \in \mathbb{N}, f {\rm smooth}\,\, {\rm function}\}$, \textit{i.e.},\\

$$A: \mathcal{C}^{\infty} \rightarrow \mathcal{C}.$$
\end{remark}



\subsection{Categorial Logic and classifying topoi}

In this subsection we list the main logical-categorial notions and results that we will need in the sequel of this work. The main references here are \cite{ShvLog}, \cite{Borceux3}, \cite{Borceux2} and \cite{MakkaiReyes}.\\

{\bf (I) Sketches and their models:}\\

$\bullet$ A (small) sketch is a 4-tuple $\mathcal{S} = (G, D, P, I)$ (\cite{Borceux2}), where $G$ is a (small) oriented graph; $D$ is a  (set)class of small (non-commutative) diagrams over $G$; $P$ is a (set)class of (non-commutative) cones over $G$; $I$ is a (set)class of (non-commutative) co-cones over $G$.  $\cal S$ is a geometric sketch if $P$ is a set of cones over $G$ with finite basis. Each (small) category $C$ determines a (small) sketch: ${\rm sk}({\cal C}) = (|{\cal C}|, D_{\cal C}, P_{\cal C}, I_{\cal C})$, where $|{\cal C}|$ is the underlying graph of the category, $D_{\cal C}$ is the class of all small commutative over ${\cal C}$, $P_{\cal C}$ is the class of all small limit cones over $\cal C$, $I_{\cal C}$ is the class of all small colimit co-cones over $\cal C$. A sketch $\mathcal{S} = (G, D, P, I)$ is called a (${\cal P}, {\cal I}$)-type if the base of all cones in $P$ are in the class $\cal P$ and   if the base of all co-cones in $I$ are in the class $\cal I$. \\

$\bullet$ A morphism of sketches ${\cal S} \to {\cal S'}$ is a homomorphism of the underlying graphs that preserves all the  given structures. This determines a (very large) category ${\rm SK}$. \\

$\bullet$ A model of a sketch $\cal S$ in a category $\cal C$ is a morphism of sketches ${\cal S} \to {\rm sk}(C)$. We will denote ${\rm Mod}({\cal S},{\cal C})$ the category whose objects are the models of $\cal S$ into the category $\cal C$ and the arrows are the natural transformations between the models (this makes sense since $\cal C$ is a category). Many usual categories of (first-order, but not necessarily finitary) mathematical structures $\cal K$ can be described as  ${\cal K} \simeq {\rm Mod}({\cal S}, {\rm \bf Set}) = {\rm SK}({\cal S},{\rm sk}({\rm \bf Set}))$ for some small sketch $\cal S$;   for instance: groups and their homomorphisms, rings and their homomorphisms, fields and their homomorphisms, local rings and local homomorphisms,  $\sigma$-boolean algebras and their homomorphisms, Banach spaces and linear contractions.\\

$\bullet$ Every small sketch $\cal S$ of  (${\cal P}, {\cal I}$)-type has a ``canonical'' (${\cal P}, {\cal I}$)-model $M : {\cal S} \to {\rm sk}(\hat{\cal S})$, where $\hat{\cal S}$ is a  ${\cal P}$-complete and ${\cal I}$-cocomplete category called ``the (${\cal P}, {\cal I}$)-theory of $\cal S$''. That is, it has all limits of the type occurring. This means that for each category $\cal C$  that is ${\cal P}$-complete and ${\cal I}$-cocomplete  composing with $M$ yields an equivalence of categories ${\rm Func}_{({\cal P}, {\cal I})} (\hat{\cal S}, {\cal C}) \overset{\simeq}\to {\rm Mod}({\cal S}, {\cal C}) = {\rm SK}({\cal S}, {\rm sk}({\cal C}))$, where ${\rm Func}_{({\cal P}, {\cal I})} (\hat{\cal S}, {\cal C})$ is the full subcategory of ${\rm Func}(\hat{\cal S}, {\cal C})$, of all functors that preserves ${\cal P}$-limits and ${\cal I}$-colimits. The (${\cal P}, {\cal I}$)-theory $\hat{\cal S}$ is unique up to ``equivalence of categories''. \\

{\bf (II) Grothendieck Topoi and geometric morphisms:}\\

$\bullet$ A (small) site is a pair $({\cal C}, J)$ formed by a (small) category ${\cal C}$ and a Grothendieck (pre)topology $J$ on $\cal C$, i.e. a map $C \in {\rm Obj}({\cal C}) \mapsto J(C)$ where $f \in J(C)$ is a small family of $\cal C$-arrows ${\cal F}=\{ f_i : A_i \to C\}_{i \in I}$  that  satisfies: the isomorphism axiom; stability axiom and transitivity axiom (\cite{ShvLog}). The usual notion of covering by opens in a topological space $X$ provides a site $({\rm Open}(X), J)$.\\

$\bullet$ Similar to the case of (pre)sheaves over a topological space  it can be defined in general the (pre)sheaves category: ${\rm Sh}({\cal C},J)  \hookrightarrow {\rm \bf Set}^{{\cal C}^{\rm op}}$ and the sheafification (left adjoint) functor $a : {\rm \bf Set}^{{\cal C}^{\rm op}} \to {\rm Sh}({\cal C},J)$ : determines a geometric morphism.\\

$\bullet$ A Grothendieck topos ${\cal E}$ is a category that is equivalent to the category of sheaves over a small site $({\cal C}, J)$,  ${\cal E} \simeq {\rm Sh}({\cal C}, J) \hookrightarrow {\rm \bf Set}^{({\cal C}^{\rm op})}$.\\

$\bullet$ A geometric morphism between the Grothendieck topoi ${\cal E}, {\cal E}'$,  $f: {\cal E} \to {\cal E}'$, is a functor $f^* : {\cal E}' \to {\cal E}$ that preserves small colimits and is left exact (i.e. it preserves finite limits). Equivalently a geometric morphism ${\cal E} \to {\cal E}'$ a is an equivalent class of adjoint functors
$$ {\cal E} \overset{f_*}{\underset{f^*}\rightleftarrows} {\cal E}'$$
where $f^*$ is left exact and left adjoint to $f_*$, and $(f^*, f_*) \equiv (g^*, g_*)$ iff $f^* = g^*$ ( and thus $f_* \cong g_*$). If $({\cal C},J)$ is a small site, the ``sheafification (left adjoint) functor'' $a : {\rm \bf Set}^{{\cal C}^{\rm op}} \to {\rm Sh}({\cal C},J)$  determines a geometric morphism ${\rm Sh}({\cal C},J)  \to {\rm \bf Set}^{{\cal C}^{\rm op}}$.\\

$\bullet$ If ${\cal E},{\cal F}$ are Grothendieck topoi,  we denote ${\rm Geom}({\cal F}, {\cal E}) \hookrightarrow {\rm Func}({\cal E},{\cal F})$ the full subcategory of the category of functors
and natural transformations formed by the (left adjoint part) of geometric morphisms $F \to E$.\\

{\bf (III) (Functorial) Theories:}\\

$\bullet$ A mathematical theory $T$ will be called a {\em functorial} mathematical theory, when there is a small category ${\cal C}_T$ such that the category of models of this theory in a Grothendieck topos $\cal E$, ${\rm Mod}_{\cal E}(T)$ is (naturally) equivalent to a full subcategory of ${\rm Hom}_T({\cal C}_T,{\cal E}) \hookrightarrow {\rm Func}({\cal C}_T,{\cal E})$. This category ${\cal C}_T$ is unique up to equivalence.\\

$\bullet$ Let $\cal C$ be a small  category with finite products and consider the (functorial) theory of finite product preserving  functors on $\cal C$, i.e. ${\cal C}_T = C$ and ${\rm Mod}_T({\cal E}) = {\rm Hom}_T({\cal C}_T, {\cal E}) = {\rm Prod}_{\rm fin}({\cal C}_T,{\cal E}) \hookrightarrow {\rm Func}({\cal C}, {\cal E})$.\\

$\bullet$ Let $\cal C$ be a small left exact category (i.e. $\cal C$ has all finite limits) and consider the (functorial) theory of left exact functors (= finite limits preserving functors) on $\cal C$, i.e. ${\cal C}_T = {\cal C}$ and ${\rm Mod}_T({\cal E}) = {\rm Hom}_T({\cal C}_T, {\cal E}) = {\rm \bf Lex}({\cal C}_T,{\cal E}) \hookrightarrow {\rm Func}({\cal C}, {\cal E})$.\\

$\bullet$ Examples of functorial mathematical theories are given by the theories $\hat{\cal S}$ associated to small sketches ${\cal S} = (G, D, P, I)$ (see (I) above).\\

$\bullet$  To each geometric/coherent first-order theory in the infinitary language  $L_{\infty\omega}$ can be associate a small ``syntactical'' category ${\cal C}_T$  in such a way to determine a functorial theory (\cite{MakkaiReyes}). \\

{\bf (IV) Classifying topoi:}\\

$\bullet$ Let $T$ be a functorial mathematical theory. $T$ admits a classifying topos when  there are (i) a Grothendieck topos ${\cal E}(T)$; (ii) a model $M : {\cal C}_T \to {\cal E}(T)$; that are (2-)universal in the following sense: given a Grothendieck topos ${\cal F}$, composing $M$ with the left adjoint part of the geometric morphism yields an equivalence of categories ${\rm Geom}({\cal F}, {\cal E}[T]) \overset{\simeq}\to {\rm Hom}_T({\cal C}_T,{\cal F})$.
The topos ${\cal E}[T]$ is called {\em the} classifying topos of the theory $T$ and the model $M$ is called {\em the} generic model of the theory $T$.\\

$\bullet$ Each classifying topos of a functorial mathematical theory determines an equivalence of categories ${\rm Geom}({\cal F}, {\cal E}[T]) \simeq {\rm Mod}_F(T)$, for each Grothendieck topos $\cal F$.  When a functorial mathematical theory admits is a classifying topos, it is unique up to equivalence of categories.\\

$\bullet$ Let $\cal C$ be a small left exact category, then  the theory of left exact functors on $\cal C$ admits the presheaves category ${\rm \bf Set}^{{\cal C}^{\rm op}}$ as a classifying topos and the Yoneda embedding $Y_{\cal C} : {\cal C} \to {\rm \bf Set}^{{\cal C}^{\rm op}}$ is the generic model.\\

$\bullet$ If $({\cal C},J)$ is a small site over a left exact category $\cal C$, then the theory of left-exact (i.e. finite limit preserving) continuous (i.e. takes covering into colimits) functors is classified by the topos ${\rm Sh}({\cal C},J)$, where the canonical model is ${\cal C} \overset{Y}\to {\rm \bf Set}^{{\cal C}^{op}} \overset{a}\to {\rm Sh}({\cal C},J)$. This includes the previous case of presheaves  categories, by taking the Grothendieck topology $J(c) = \{{\rm id}_{\cal C}   : {\cal C} \to {\cal C}\}, c \in {\rm Obj}({\cal C})$.\\

$\bullet$ If the small category ${\cal C}_T$ that encodes a mathematical theory $T$ is freely generated by an object $u$, then the generic model  $M : {\cal C}_T \to  {\rm Sh}({\cal C}_T,J)$ is uniquely determined (up to natural isomorphism) by $M(u) = a({\cal C}_T(-,u))$. In this case, ${\rm ev}_u: {\rm Hom}_{T}({\cal C}_T, {\cal E}) \stackrel{\simeq}{\rightarrow} {\rm Mod}_{{\cal E}}(T)$ is an equivalence of categories for each Grothendieck topos ${\cal E}$. Such object $u$ is called the ``universal'' $T$-object in ${\rm Sh}({\cal C}_T,J)$.\\

$\bullet$ The Mitchell-B\'enabou language of a elementary/Grothendieck topos and the Kripke-Joyal semantics allows us to interpret --in particular--  first-order formulas in many sorted languages $L_{\omega \omega}/ L_{\infty \omega}$ in a  elementary/Grothendieck topos. Every geometric theory admits a classifying topos.\\

$\bullet$ Every Grothendieck topos is the classifying topos of a small geometric sketch.

\section{A Classifying Topos for the Theory of  $C^{\infty}-$rings}

In this section we describe a classifying topos for the theory of $\mathcal{C}^{\infty}-$rings. We mimic the construction of a classifying topos for the theory of commutative unital rings, given by I. Moerdijk and S. Mac Lane in \cite{ShvLog}, making some necessary adaptations to the context of $\mathcal{C}^{\infty}-$rings.\\

\subsection{$\mathcal{C}^{\infty}-$Ring Objects in Categories with Finite Products}

\begin{definition}\label{Gabrielle}Let $\mathcal{C}$ be a category with finite products. A \textbf{$\mathcal{C}^{\infty}-$ring object} in $\mathcal{C}$ is a morphism of sketches $A: \mathcal{S}_{\mathcal{C}^{\infty}{\rm \bf Rng}} \to {\rm sk}(\mathcal{C})$, where $\mathcal{S}_{\mathcal{C}^{\infty}{\rm \bf Rng}}$ is the sketch of the theory of $\mathcal{C}^{\infty}-$rings.
\end{definition}

\begin{proposition}\label{Maricruz}Given a $\mathcal{C}^{\infty}-$ring-object $A :\mathcal{S}_{\mathcal{C}^{\infty}-{\rm \bf Rng}} \to {\rm sk}(\mathcal{C})$ in $\mathcal{C}$,in the sense of the \textbf{Definition \ref{Gabrielle}}, the object $A(|\mathbb{R}|) \in {\rm Obj}\,(\mathcal{C})$ has an obvious $\mathcal{C}^{\infty}-$ring structure, $\Psi$, given by:

$$\begin{array}{cccc}
    \Psi: & \bigcup_{n \in \mathbb{N}}\mathcal{C}^{\infty}(\mathbb{R}^n, \mathbb{R}) & \rightarrow & \bigcup_{n \in \mathbb{N}} {\rm Hom}_{\mathcal{C}}(A(|\mathbb{R}|)^{n},A(|\mathbb{R}|)) \\
     & f & \mapsto & A(|f|): A(|\mathbb{R}|)^n \to A(|\mathbb{R}|)
  \end{array}$$

Thus, we have the (universal-algebraic) $\mathcal{C}^{\infty}-$ring $(A(|\mathbb{R}|), \Psi)$.
\end{proposition}
\begin{proof}
It suffices to prove that $\Psi$ satisfies the two groups of axioms given in \textbf{Definition \ref{CravoeCanela}}.\\

$\Psi$ preserves projections, since $A$, as a $\mathcal{C}^{\infty}-$ring object, maps the projective cones given in $\mathcal{P}$ to limit cones in $\mathcal{C}$ - that is, to products. Given $n, m_1, \cdots, m_k \in \mathbb{N}$ such that $n = m_1 + \cdots, m_k$ and the projections $p^{n}_{m_i}: \mathbb{R}^n \to \mathbb{R}^{m_i}, i=1, \cdots, k$, $\Psi(p^{n}_{m_i}):= A(|p^{n}_{m_i}|): A(|\mathbb{R}|)^n \to A(|\mathbb{R}|)^{m_i}$, which must be the projections since $A$ maps the cone $(p^{n}_{m_i}: \mathbb{R}^n \to \mathbb{R}^{m_i})_{i=1, \cdots, k}$ to a product in $\mathcal{C}$.\\

Also, for every $n \in \mathbb{N}$ and every $(n+2)-$tuple of $\mathcal{C}^{\infty}-$functions, $(h,g_1, \cdots, g_n, f)$
with $f\in \mathcal{C}^{\infty}(\mathbb{R}^n)$, $g_1, \cdots, g_n \in \mathcal{C}^{\infty}(\mathbb{R}^k)$ with:
$$h = f \circ (g_1, \cdots, g_n)$$
we have:

$$\Psi(f \circ (g_1, \cdots, g_n)) = A(|f| \circ (|g_1|, \cdots, |g_n|)) = A(|f|) \circ A((|g_1|, \cdots, |g_n|)),$$
since $A$, as a $\mathcal{C}^{\infty}-$ring, maps the diagram:

$$\xymatrixcolsep{5pc}\xymatrix{
|\mathbb{R}^k| \ar[r]^{(|g_1|, \cdots, |g_n|)} \ar[dr]_{|h|} & |\mathbb{R}^n| \ar[d]^{|f|}\\
 & |\mathbb{R}|
}$$

(that belongs to $\mathcal{D}$ since $h = f \circ (g_1, \cdots,g_n)$) to a commutative one:

$$\xymatrixcolsep{5pc}\xymatrix{
A(|\mathbb{R}^k|) \ar[dr]_{A(|h|)} \ar[r]^{A((|g_1|, \cdots, |g_n|))} & A(|\mathbb{R}^n|)\ar[d]^{A(|f|)} \\
 & A(|\mathbb{R}|)}$$

that is $A(|h|) = A(|f|) \circ A((|g_1|, \cdots, |g_n|))$.\\

\textbf{Claim:} $A((|g_1|, \cdots, |g_n|)) = (A(|g_1|), \cdots, A(|g_n|))$.\\

Indeed, for every $i \in \{ 1, \cdots, k\}$ the following diagram commutes:

$$\xymatrixcolsep{5pc}\xymatrix{
A(|\mathbb{R}^k|) \ar[dr]_{A(|g_i|)} \ar[r]^{A((|g_1|, \cdots, |g_n|))} & A(|\mathbb{R}^n|) \ar[d]^{A(|p^{n}_{i}|)}\\
  & A(|\mathbb{R}|)
}$$

and since $A$ interprets each $p^{n}_{i}$, $i=1, \cdots,k$, as a projection, $A(|p^{n}_{i}|)$, it follows that:

$$A((|g_1|, \cdots, |g_n|)) = (A(|g_1|), \cdots, A(|g_n|)).$$

Thus

\begin{multline*}
\Psi(h) := A(|h|) = A(|f|\circ (|g_1|, \cdots, |g_n|)) = A(|f|) \circ A((|g_1|, \cdots, |g_n|))=\\
=A(|f|)\circ (A(|g_1|), \cdots, A(|g_n|)) = \Psi(f)\circ (\Psi(g_1), \cdots, \Psi(g_n))
\end{multline*}

and $\Psi$ is a $\mathcal{C}^{\infty}-$ring structure.
\end{proof}

\begin{remark}Let $\mathcal{C}$ be a category with all finite limits. The category $\underline{\mathcal{C}^{\infty}-{\rm Ring}}\,(\mathcal{C})$ is not a subcategory of $\mathcal{C}$ (cf. p. 101 of \cite{Schubert}). However, there is a forgetful functor $U: \underline{\mathcal{C}^{\infty}-{\rm Ring}}\,(\mathcal{C}) \to \mathcal{C}$ which is faithfull and reflects isomorphisms (cf. \textbf{Proposition 11.3.3} of \cite{Schubert}). It follows that $U$ reflects all the limits and colimits that it preserves and which exist in $\underline{\mathcal{C}^{\infty}-{\rm Ring}}\,(\mathcal{C})$.\\
 \end{remark}

 The following proposition gives us some properties of the category $\underline{\mathcal{C}^{\infty}-{\rm Ring}}\,(\mathcal{C})$ which are inherited from $\mathcal{C}$,\\

\begin{proposition}\label{Schubert}If a category $\mathcal{C}$ is finitely complete, then the same is true for the category $\mathcal{C}^{\infty}-{\rm Ring}\,(\mathcal{C})$.
 \end{proposition}
\begin{proof}
It is an immediate application of \textbf{Proposition 11.5.1} of page 103 of \cite{Schubert}.
\end{proof}

\begin{proposition}Let $\mathcal{C}$ be a category with all finite limits. Every left-exact functor $F: \mathcal{C} \to \mathcal{C}'$ induces a functor:
$$T_{\underline{\mathcal{C}^{\infty}-{\rm Ring}}}: \underline{\mathcal{C}^{\infty}-{\rm Ring}}\,(\mathcal{C}) \to \underline{\mathcal{C}^{\infty}-{\rm Ring}}\,(\mathcal{C}')$$
\end{proposition}
\begin{proof}
Since every functor preserves commutative diagrams, it follows that $F$ maps commutative diagrams of $\mathcal{C}$ to commutative diagrams of $\mathcal{C}'$, so the $\mathcal{C}^{\infty}-$ring-objects of $\mathcal{C}$ are mapped to $\mathcal{C}^{\infty}-$ring-objects of $\mathcal{C}'$.
\end{proof}

\begin{proposition}The object $\mathcal{C}^{\infty}(\mathbb{R})$ of $\mathcal{C}^{\infty}{\rm \bf Rng}_{{\rm fp}}$ is a $\mathcal{C}^{\infty}-$\textbf{ring-object} in
$$\mathcal{C}^{\infty}{\rm \bf Rng}_{{\rm fp}},$$
\end{proposition}
\begin{proof}Given any $f \in \mathcal{C}^{\infty}(\mathbb{R}^n, \mathbb{R}) \subseteq \bigcup_{n \geq 0} \mathcal{C}^{\infty}(\mathbb{R}^n,\mathbb{R})$ we define $\widehat{f}$ as the unique \\ $\mathcal{C}^{\infty}-$homomorphism sending the identity function ${\rm id}_{\mathbb{R}}: \mathbb{R} \to \mathbb{R}$ to $f$, that is:
$$\begin{array}{cccc}
    \widehat{f} = f \circ - : & \mathcal{C}^{\infty}(\mathbb{R}) & \rightarrow & \mathcal{C}^{\infty}(\mathbb{R}^n) \\
     & g & \mapsto & f \circ g
  \end{array}$$
\end{proof}

\begin{theorem}The category $$\mathcal{C}^{\infty}{\rm \bf Rng}_{{\rm fp}}^{{\rm op}}$$ is a category with finite limits freely generated by the $\mathcal{C}^{\infty}-$ring-object $\mathcal{C}^{\infty}(\mathbb{R})$.
\end{theorem}
\begin{proof}
As we have already commented, this amounts to prove that  for any category with finite limits, $\mathcal{C}$, the evaluation of a left-exact functor $F: \mathcal{C}^{\infty}{\rm \bf Rng}_{{\rm fp}}^{{\rm op}} \to \mathcal{C}$ at $\mathcal{C}^{\infty}(\mathbb{R})$ gives the following equivalence of categories:

$$\begin{array}{cccc}
    {\rm ev}_{\mathcal{C}^{\infty}(\mathbb{R})}: & {\rm \bf Lex}\,(\mathcal{C}^{\infty}{\rm \bf Rng}_{{\rm fp}}^{{\rm op}}, \mathcal{C}) & \rightarrow & \underline{\mathcal{C}^{\infty}-{\rm Rings}}\,(\mathcal{C}) \\
     & F & \mapsto & F(\mathcal{C}^{\infty}(\mathbb{R}))
  \end{array}$$

First note that this correspondence is indeed a function, for if $F$ is left-exact, then it preserves $\mathcal{C}^{\infty}-$ring-objects, hence it sends the $\mathcal{C}^{\infty}-$ring object $\mathcal{C}^{\infty}(\mathbb{R})$ of $\mathcal{C}^{\infty}{\rm \bf Rng}_{{\rm fp}}^{{\rm op}}$ into a $\mathcal{C}^{\infty}-$ring object of $\mathcal{C}$.\\

We are going to show that this functor is full, faithful and dense.\\

$\bullet$ ${\rm ev}_{\mathcal{C}^{\infty}(\mathbb{R})}$ \textbf{ is faithful};\\

Let $F, G \in {\rm Obj}\,({\rm \bf Lex}\,(\mathcal{C}^{\infty}{\rm \bf Rng}_{{\rm fp}}^{{\rm op}}, \mathcal{C}))$ and let $\eta, \theta: F \Rightarrow G$ be two natural transformations between them such that:

$$(\eta_{\mathcal{C}^{\infty}(\mathbb{R})}: F(\mathcal{C}^{\infty}(\mathbb{R})) \rightarrow G(\mathcal{C}^{\infty}(\mathbb{R}))) = (\theta_{\mathcal{C}^{\infty}(\mathbb{R})}: F(\mathcal{C}^{\infty}(\mathbb{R})) \rightarrow G(\mathcal{C}^{\infty}(\mathbb{R}))).$$

We prove that given any object $A$ of $\mathcal{C}^{\infty}-{\rm \bf Rng}_{{\rm fp}}^{{\rm op}}$, we have $\eta_A = \theta_A$.\\

First suppose $A = \mathcal{C}^{\infty}(\mathbb{R}^n)$, that is, $A = \mathcal{C}^{\infty}(\mathbb{R})\otimes_{\infty}\cdots \otimes_{\infty}\mathcal{C}^{\infty}(\mathbb{R})$ (which is a product in $\mathcal{C}^{\infty}-{\rm \bf Rng}_{{\rm fp}}^{{\rm op}}$). Since $F$ is left-exact, $F(\mathcal{C}^{\infty}(\mathbb{R}^n)) = F(\mathcal{C}^{\infty}(\mathbb{R}))^n$, and:

$$\eta_{\mathcal{C}^{\infty}(\mathbb{R}^n)}= \eta_{\mathcal{C}^{\infty}(\mathbb{R})} \times \cdots \times \eta_{\mathcal{C}^{\infty}(\mathbb{R})}: F(\mathcal{C}^{\infty}(\mathbb{R}))^n \rightarrow G(\mathcal{C}^{\infty}(\mathbb{R}))^n$$

Since $\eta_{\mathcal{C}^{\infty}(\mathbb{R})}=\theta_{\mathcal{C}^{\infty}(\mathbb{R})}$, it follows that $\eta_{\mathcal{C}^{\infty}(\mathbb{R}^n)} = \theta_{\mathcal{C}^{\infty}(\mathbb{R}^n)}$.\\

$\bullet$ ${\rm ev}_{\mathcal{C}^{\infty}(\mathbb{R})}$ \textbf{is  full;}\\

Let $F,G \in {\rm Obj}\,({\rm \bf Lex}\,(\mathcal{C}^{\infty}-{\rm \bf Rng}_{{\rm fp}}^{{\rm op}}, \mathcal{C}))$ and let $\varphi: F(\mathcal{C}^{\infty}(\mathbb{R})) \rightarrow G(\mathcal{C}^{\infty}(\mathbb{R}))$ be a morphism in $\underline{\mathcal{C}^{\infty}-{\rm Rings}}\,(\mathcal{C})$. It suffices to take $\eta: F \Rightarrow G$ such that $\eta_{\mathcal{C}^{\infty}(\mathbb{R})} = \varphi$.\\

$\bullet$ ${\rm ev}_{\mathcal{C}^{\infty}(\mathbb{R})}$ \textbf{is isomorphism dense;}\\

Let $R$ be any object in $\underline{\mathcal{C}^{\infty}-{\rm Rings}}\,(\mathcal{C})$.\\

Given this object $R$, we are going to construct $\phi_R \in {\rm Obj}\,({\rm \bf Lex}\,(\mathcal{C}^{\infty}-{\rm \bf Rng}_{\rm fp}^{\rm op}, \mathcal{C}))$ such that ${\rm ev}_{\mathcal{C}^{\infty}(\mathbb{R})}(\phi_R) \cong R$.\\

We set $\phi_R(\mathcal{C}^{\infty}(\mathbb{R}))=R$.\\

We first define the action of $\phi_R$ on the free $\mathcal{C}^{\infty}-$ring objects.\\

Now, given a free $\mathcal{C}^{\infty}-$ring o bject on $n$ generators, $R^n$ , since $\phi_R$ is to be left-exact, it transforms coproducts in $\mathcal{C}^{\infty}-{\rm \bf Rng}_{{\rm fp}}$ into products of $\mathcal{C}$. Hence, since $\mathcal{C}^{\infty}(\mathbb{R}^n) \cong \mathcal{C}^{\infty}(\mathbb{R})\otimes_{\infty} \cdots \otimes_{\infty}\mathcal{C}^{\infty}(\mathbb{R})$, we set:

$$\phi_R(\mathcal{C}^{\infty}(\mathbb{R}^n)) = R^n,$$

which establishes the action of $\phi_R$ on the free objects of $\mathcal{C}^{\infty}{\rm \bf Rng}_{{\rm fp}}^{{\rm op}}$.\\

Now we shall describe the action of $\phi_R$ on the arrows between objects of $\mathcal{C}^{\infty}{\rm \bf Rng}_{{\rm fp}}^{{\rm op}}$:

$$\begin{array}{cccc}
    (\phi_R)_1: & {\rm Mor}\,(\underline{\mathcal{C}^{\infty}-{\rm Rngs}}\,(\mathcal{C})) & \rightarrow & {\rm Nat}\,({\rm \bf Lex}\,(\mathcal{C}^{\infty}-{\rm \bf Rng}_{{\rm fp}}^{{\rm op}}, \mathcal{C}))
\end{array}$$

beginning with the $\mathcal{C}^{\infty}-$homomorphisms between the free objects of $\mathcal{C}^{\infty}-{\rm \bf Rng}_{{\rm fp}}^{{\rm op}}$.\\

An arrow (i.e., a $\mathcal{C}^{\infty}-$homomorphism) in $\mathcal{C}^{\infty}{\rm \bf Rng}_{{\rm fp}}$ between free $\mathcal{C}^{\infty}-$rings is a map:
$$\begin{array}{cccc}
    p: & \mathcal{C}^{\infty}(\mathbb{R}^k) & \rightarrow & \mathcal{C}^{\infty}(\mathbb{R}^n) \\
     & (\xymatrix{\mathbb{R}^k \ar[r]^{g} & \mathbb{R}}) & \mapsto & (\xymatrix{\mathbb{R}^n \ar[r]^{{\bf p}(g)} & \mathbb{R}})
  \end{array}$$

given by a $k-$tuple of smooth functions, $(p_1, \cdots, p_k): \mathbb{R}^n \to \mathbb{R}^k$:

$$\xymatrixcolsep{5pc}\xymatrix{
\mathbb{R}^k  \ar[dr]_{g} & & \mathbb{R}^n \ar[ll]_{(p_1, \cdots, p_k)} \ar@{-->}[dl]^{p(g)}\\
  & \mathbb{R}&
}$$

where $p_i = p(\pi_i): \mathbb{R}^n \to \mathbb{R}, i=1, \cdots, k$ and $\pi_i: \mathbb{R}^k \to \mathbb{R}$ is the projection on the $i-$th coordinate.\\

Each such smooth function $p_i: \mathbb{R}^n \to \mathbb{R}$ yields an arrow in $\mathcal{C}$:

\begin{equation}\label{13}
p_i^{(R)}: R^n \to R
\end{equation}

defined from the $\mathcal{C}^{\infty}-$ring structure (defined in the \textbf{Proposition \ref{Maricruz}}), say $\Psi$, of $R \in \underline{\mathcal{C}^{\infty}-{\rm Rings}}\,(\mathcal{C})$, which interprets  every smooth function in $\mathcal{C}$.\\

We have, as a direct consequence of the fact pointed out by Moerdijk and Reyes on the page 21 of \cite{MSIA}, a $1-1$ correspondence between $\mathcal{C}^{\infty}-$homomorphisms from $\mathcal{C}^{\infty}(\mathbb{R}^k)$ to $\mathcal{C}^{\infty}(\mathbb{R}^n)$ and $k-$tuples of smooth functions from $\mathbb{R}^n$ to $\mathbb{R}$:

$$\dfrac{p: \mathcal{C}^{\infty}(\mathbb{R}^k) \rightarrow \mathcal{C}^{\infty}(\mathbb{R}^n)}{\mathbb{R}^n \stackrel{(p_1, \cdots, p_k)}{\longrightarrow} \mathbb{R}^{k}}.$$







The image under $\phi_R$ of the arrow $p: \mathcal{C}^{\infty}(\mathbb{R}^k) \to \mathcal{C}^{\infty}(\mathbb{R}^n)$ is calculated first taking the $k$-tuple of smooth functions given by the correspondence:

$$\dfrac{p: \mathcal{C}^{\infty}(\mathbb{R}^k) \rightarrow \mathcal{C}^{\infty}(\mathbb{R}^n)}{\mathbb{R}^n \stackrel{(p_1, \cdots, p_k)}{\longrightarrow} \mathbb{R}^{k}}.$$

and then interpreting it in $R$:
\begin{equation}\label{14}
\phi_R(\xymatrix{\mathcal{C}^{\infty}(\mathbb{R}^k) \ar[r]^{p} & \mathcal{C}^{\infty}(\mathbb{R}^n)}) = p^{(R)} = (p_1^{(R)}, \cdots, p_k^{(R)}): R^n \to R^k
\end{equation}

To complete the definition of the functor $\phi_R$ on any finitely presented $\mathcal{C}^{\infty}-$ring $\dfrac{\mathcal{C}^{\infty}(\mathbb{R}^n)}{\langle p_1, \cdots, p_k\rangle}$, we note that, by definition, this quotient fits into a coequalizer diagram:

\begin{equation}\label{15}
\xymatrixcolsep{5pc}\xymatrix{ \mathcal{C}^{\infty}(\mathbb{R}^k) \ar[r]^{p} \ar@<-1ex>[r]_{0} & \mathcal{C}^{\infty}(\mathbb{R}^n) \ar@{->>}[r]^{q_{\langle p_1, \cdots, p_k\rangle}} & \dfrac{\mathcal{C}^{\infty}(\mathbb{R}^n)}{\langle p_1, \cdots, p_k\rangle}   }
\end{equation}

where $p_i = p(\pi_i)$ for $i=1, \cdots, k$ and $0(\pi_i) = 0$ for $i=1, \cdots,k$.\\

The category $\mathcal{C}$, by hypothesis, has all finite limits, so the category of the $\mathcal{C}^{\infty}-$rings objects in a category $\mathcal{C}$ has equalizers, and there is an equalizer diagram:

$$\xymatrix{E \,\, \ar@{>->}[r]^{e} & R^n \ar[r]^{p^{(R)}} \ar@<-1ex>[r]_{0^{(R)}} & R^m}$$

Thus we define the image under the contravariant functor $\phi_R$ of the finitely presented $\mathcal{C}^{\infty}-$ring $\dfrac{\mathcal{C}^{\infty}(\mathbb{R}^n)}{\langle p_1, \cdots, p_k\rangle}$ as:

$$\phi_R \left(\dfrac{\mathcal{C}^{\infty}(\mathbb{R}^n)}{\langle p_1, \cdots, p_k\rangle} \right):=E$$

that is, by the following equalizer diagram in $\mathcal{C}$:

\begin{equation}\label{16}
\xymatrixcolsep{5pc}\xymatrix{ \phi_R\left( \dfrac{\mathcal{C}^{\infty}(\mathbb{R}^n)}{\langle p_1, \cdots, p_k\rangle}\right) \,\, \ar@{>->}[r] & R^n \ar[r]^{p^{(R)}} \ar@<-1ex>[r]_{0^{(R)}} & R^k}
\end{equation}

Next, we define $\phi_R$ on a $\mathcal{C}^{\infty}-$homomorphism $h: B \to C$ between any two finitely presented $\mathcal{C}^{\infty}-$rings. Let $\dfrac{\mathcal{C}^{\infty}(\mathbb{R}^n)}{\langle p_1, \cdots, p_k\rangle}$ and $\dfrac{\mathcal{C}^{\infty}(\mathbb{R}^m)}{\langle g_1, \cdots, g_t\rangle}$ be two finitely presented $\mathcal{C}^{\infty}-$rings and let:

$$\dfrac{\mathcal{C}^{\infty}(\mathbb{R}^n)}{\langle f_1, \cdots, f_k\rangle} \stackrel{\Phi}{\rightarrow} \dfrac{\mathcal{C}^{\infty}(\mathbb{R}^m)}{\langle g_1, \cdots, g_t\rangle}$$

be a $\mathcal{C}^{\infty}-$homomorphism. The $\mathcal{C}^{\infty}-$homomorphism $\Phi$ is determined by some $\mathcal{C}^{\infty}-$function:

$$\begin{array}{cccc}
\varphi: & \mathbb{R}^m & \rightarrow & \mathbb{R}^n\\
         & x & \mapsto & (\varphi_1(x), \cdots, \varphi_n(x))
\end{array}$$

such that $\langle f_1, \cdots, f_k\rangle \subseteq \varphi_*[\langle g_1, \cdots, g_t\rangle]$. Hence, the $\mathcal{C}^{\infty}-$homomorphism $\Phi$ is determined by the equivalence classes of $n$ $\mathcal{C}^{\infty}-$functions: $\varphi_1, \cdots, \varphi_n: \mathbb{R}^m \to \mathbb{R}$ such that:

\begin{equation}\label{17}(\forall j \in \{1, \cdots,k \})(f_j \circ \varphi = f_j \circ (\varphi_1, \cdots, \varphi_n) \in \langle g_1, \cdots, g_t \rangle).
\end{equation}

As in \eqref{14}, these $n$ smooth functions determine a $\mathcal{C}^{\infty}-$homomorhpism $\varphi^{(R)}: R^m \to R^n$. Now $\phi_R\left(\dfrac{\mathcal{C}^{\infty}(\mathbb{R}^n)}{\langle f_1, \cdots, f_k\rangle}\right)$ and $\phi_R\left(\dfrac{\mathcal{C}^{\infty}(\mathbb{R}^m)}{\langle g_1, \cdots, g_t\rangle}\right)$ fit into equalizer rows:

$$\xymatrixcolsep{5pc}\xymatrix{
\phi_R\left(\dfrac{\mathcal{C}^{\infty}(\mathbb{R}^n)}{\langle f_1, \cdots, f_k\rangle}\right)\,\, \ar@{>->}[r] & R^n \ar[r]^{f^{(R)}} \ar@<-1ex>[r]_{0^{(R)}} & R^k \\
\phi_R\left(\dfrac{\mathcal{C}^{\infty}(\mathbb{R}^m)}{\langle g_1, \cdots, g_t\rangle}\right) \,\, \ar@{>->}[r]_{\alpha} & R^m \ar[u]^{\varphi^{(R)}}\ar[r]^{g^{(R)}} \ar@<-1ex>[r]_{0^{(R)}} & R^t
}$$

where $f^{(R)}: R^n \to R^k$ is the interpretation of $f= (f_1, \cdots, f_k): \mathbb{R}^n \to \mathbb{R}^k$,  $g^{(R)}: R^n \to R^t$ is the interpretation of $g = (g_1, \cdots, g_t): \mathbb{R}^m \to \mathbb{R}^t$, the equalizer $\alpha$ in the lower left is determined by $m$ arrows, $\alpha = (\alpha_1, \cdots, \alpha_m)$ with

$$\alpha_s: \phi_R\left(\dfrac{\mathcal{C}^{\infty}(\mathbb{R}^m)}{\langle g_1, \cdots, g_t\rangle}\right) \to R, s=1, \cdots, m$$

which satisfy (by definition) $g_{\ell}(\alpha_1, \cdots, \alpha_m) = 0$ for every $\ell \in \{1, \cdots, t \}$. \\

We have:

$$f \circ \varphi = (f_1 \circ \varphi, \cdots, f_k \circ \varphi)$$

so

$$f^{(R)} \circ \varphi^{(R)} = (f_1^{(R)} \circ \varphi^{(R)}, \cdots, f_k^{(R)} \circ \varphi^{(R)}).$$

Since for every $i=1, \cdots,k$, $f_i \circ \varphi \in \langle g_1, \cdots, g_t\rangle$, there are $\ell$ $\mu_1, \cdots, \mu_t \in \mathcal{C}^{\infty}(\mathbb{R}^m)$ such that:
$${f_i}^{(R)}\circ \varphi^{(R)} = \sum_{\ell = 1}^{t} \mu_{\ell}\cdot g_{\ell},$$

and by \eqref{17}, it follows that:

$$(\forall i \in \{1, \cdots, k \})({f_i}^{(R)}\circ \varphi^{(R)}(\alpha_1, \cdots, \alpha_m) = \sum_{\ell =1}^{t} \mu_{\ell}(\alpha_1, \cdots, \alpha_m)\cdot \underbrace{g_{\ell}(\alpha_1, \cdots, \alpha_m)}_{=0}),$$

so

$$f^{(R)} \circ (\varphi^{(R)} \circ \alpha) = ({f_1}^{(R)}\circ \varphi^{(R)}, \cdots, {f_k}^{(R)}\circ \varphi^{(R)}) = 0^{(R)}.$$

Hence, the composite $\varphi^{(R)} \circ \alpha$ consists of $n$ arrows to $R$ which satisfy the conditions $f \circ (\varphi^{(R)} \circ \alpha) = 0$.\\

Therefore, by the universal property of equalizers, there is a unique arrow $\phi_R(h)$, indicated as follows:

$$\xymatrixcolsep{5pc}\xymatrix{
\phi_R(B) \,\, \ar@{>->}[r] & R^n \ar[r]^{f^{(R)}} \ar@<-1ex>[r]_{0^{(R)}} & R^k \\
\phi_R(C) \ar@{-->}[u]^{\exists ! \phi_R(\Phi)} \ar@{>->}[ur]_{\varphi^{(R)} \circ \alpha} & &
}$$

Note that $\phi_R(\Phi)$ is independent of the choice of $\varphi_i$ in their equivalence classes, so $\phi_R$ is a functor, as required in \eqref{16}.\\

\textbf{Claim:} For each $\mathcal{C}^{\infty}-$ring object  $R$ in $\mathcal{C}$, the functor $\phi_R$ thus defined is a left-exact functor $\phi_R: \mathcal{C}^{\infty}{\rm \bf Rng}_{{\rm fp}} \to \mathcal{C}$.\\

We are going to show that $\phi_R$ preserves terminal object, binary products and equalizers, so $\phi_R$ will preserve all finite limits (which are constructed from these).\\

In fact, $\phi_R(\mathbb{R}^0)$ is the empty product of copies of $R$ [since $\phi_R(\mathbb{R}^0) = R^{0}$ for $n=0$], i.e., $\phi_R(\mathbb{R}^0)=1$, so $\phi_R$ preserves the terminal object.\\

Also, since the product of two equalizer diagrams is again an equalizer, one easily verifies from \eqref{16} that $\phi_R$ is such that for any $\dfrac{\mathcal{C}^{\infty}(\mathbb{R}^n)}{\langle f_1, \cdots, f_k\rangle}$ and any $\dfrac{\mathcal{C}^{\infty}(\mathbb{R}^m)}{\langle g_1, \cdots, g_t\rangle}$ we have:

$$\phi_R \left( \dfrac{\mathcal{C}^{\infty}(\mathbb{R}^n)}{\langle f_1, \cdots, f_k\rangle} \otimes_{\infty} \dfrac{\mathcal{C}^{\infty}(\mathbb{R}^m)}{\langle g_1, \cdots, g_t\rangle}\right) \cong \phi_R\left( \dfrac{\mathcal{C}^{\infty}(\mathbb{R}^n)}{\langle f_1, \cdots, f_k\rangle}\right) \times \phi_R \left( \dfrac{\mathcal{C}^{\infty}(\mathbb{R}^m)}{\langle g_1, \cdots, g_t\rangle}\right)$$

that is, $\phi_R$ preserves binary products.\\

Finally, to see that $\phi_R$ preserves equalizers, consider a coequalizer constructed in the evident way from two arbitrary maps $s, s'$ in the category of finitely presented $\mathcal{C}^{\infty}-$rings,

\begin{multline}\label{22}
\xymatrixcolsep{5pc}\xymatrix{\dfrac{\mathcal{C}^{\infty}(\mathbb{R}^m)}{\langle p_1, \cdots, p_k\rangle} \ar[r]^{s} \ar@<-1ex>[r]_{s'} & \dfrac{\mathcal{C}^{\infty}(\mathbb{R}^n)}{\langle g_1, \cdots, g_t\rangle} \ar@{->>}[r]& }\\
\twoheadrightarrow \dfrac{\mathcal{C}^{\infty}(\mathbb{R}^n)}{\langle g_1, \cdots, g_t, s\circ \pi_1 - s'\circ \pi_1, \cdots, s \circ \pi_k - s' \circ \pi_k \rangle }
\end{multline}

We must show that $\phi_R$ sends this coequalizer \eqref{22} to an equalizer diagram in $\mathcal{C}$.\\

First of all, if \eqref{22} is a coequalizer, then so is the diagram:

\begin{multline}\label{23}
\xymatrixcolsep{5pc}\xymatrix{ \mathcal{C}^{\infty}(\mathbb{R}^m)\ar[r]^{s \circ q_I} \ar@<-1ex>[r]_{s'\circ q_I} & \dfrac{\mathcal{C}^{\infty}(\mathbb{R}^n)}{\langle g_1, \cdots, g_t\rangle} \ar@{->>}[r] & }\\ \twoheadrightarrow \dfrac{\mathcal{C}^{\infty}(\mathbb{R}^n)}{\langle g_1, \cdots, g_t, s\circ \pi_1 \circ q_I - s'\circ \pi_1 \circ q_I, \cdots, s \circ \pi_k \circ q_I - s'\circ \pi_k \circ q_I\rangle}
\end{multline}

obtained by precomposing \eqref{22} with the epimorphism $q_I: \mathcal{C}^{\infty}(\mathbb{R}^m) \to \dfrac{\mathcal{C}^{\infty}(\mathbb{R}^m)}{\langle p_1, \cdots, p_k\rangle}$.\\

Moreover, since $\phi_R$ sends the latter epimorphism, $q_I$, to a monomorphism in $\mathcal{C}$ [in fact, to an equalizer, as in \eqref{16}, and every equalizer is a monomorphism], $\phi_R$ sends \eqref{22} to an equalizer if, and only if it does for \eqref{23}. So it suffices to show that $\phi_R$ sends coequalizers of the special form \eqref{23} to equalizers in $\mathcal{C}$.\\

Next, since \eqref{23} is a coequalizer, so is

\begin{multline}\label{24}
\xymatrixcolsep{5pc}\xymatrix{ \mathcal{C}^{\infty}(\mathbb{R}^m)\ar[r]^{s \circ q_I - s' \circ q_I} \ar@<-1ex>[r]_{s'\circ q_I} & \dfrac{\mathcal{C}^{\infty}(\mathbb{R}^n)}{\langle g_1, \cdots, g_t\rangle} \ar@{->>}[r] & } \\ \twoheadrightarrow \dfrac{\mathcal{C}^{\infty}(\mathbb{R}^n)}{\langle g_1, \cdots, g_t, s\circ \pi_1 \circ q_I - s'\circ \pi_1 \circ q_I, \cdots, s \circ \pi_k \circ q_I - s'\circ \pi_k \circ q_I\rangle}
\end{multline}

and one readly checks that $\phi_R$ sends \eqref{23} to an equalizer in $\mathcal{C}$ if, and only if it does for \eqref{24}. So, by replacing $s$ by $s - s'$ and $s'$ by $0$ in \eqref{23} we see that is suffices to show that $\phi_R$ sends coequalizers of the form \eqref{23} with $s'=0$ to equalizers in $\mathcal{C}$.\\

Given a $\mathcal{C}^{\infty}-$homomorphism $p: \mathcal{C}^{\infty}(\mathbb{R}^k) \to \mathcal{C}^{\infty}(\mathbb{R}^n)$, construct the diagram:

$$\xymatrixcolsep{5pc}\xymatrix{
   & \mathcal{C}^{\infty}(\mathbb{R}^k) \ar[d]^{0} \ar@<-1ex>[d]_{p} & \\
\mathcal{C}^{\infty}(\mathbb{R}^{m+k}) \ar[r]^{(s,p)} \ar@<-1ex>[r]_{0} \ar[d]_{\substack{\pi_i \mapsto 0\\ m+1 \leq i \leq k}} & \mathcal{C}^{\infty}(\mathbb{R}^n) \ar@{->>}[d] \ar@{->>}[r] & \dfrac{\mathcal{C}^{\infty}(\mathbb{R}^n)}{\langle p_1, \cdots, p_k, s \circ \pi_1, \cdots, s \circ \pi_m\rangle} \ar@{=}[d]\\
\mathcal{C}^{\infty}(\mathbb{R}^m) \ar[r]^{s} \ar@<-1ex>[r]_{0} & \dfrac{\mathcal{C}^{\infty}(\mathbb{R}^n)}{\langle p_1, \cdots, p_k \rangle} \ar[r] & \dfrac{\mathcal{C}^{\infty}(\mathbb{R}^n)}{\langle p_1, \cdots, p_k, s \circ \pi_1, \cdots, s \circ \pi_m\rangle}
}$$

consisting of three coequalizers, two of the form \eqref{15}. By definition \eqref{16}, $\phi_R$ sends both the vertical coequalizer and the upper horizontal coequalizer to equalizers in $\mathcal{C}$. It follows, by diagram chasing that it also sends the lower horizontal coequalizer to an equalizer in $\mathcal{C}$.\\

This shows that $\phi_R$ is a left-exact functor.\\

By construction, ${\rm ev}_{\mathcal{C}^{\infty}(\mathbb{R})}(\phi_R) = \phi_R(\mathcal{C}^{\infty}(\mathbb{R}))=R$, so ${\rm ev}_{\mathcal{C}^{\infty}(\mathbb{R})}$ is a fully faithful dense functor, hence an equivalence of categories.

\end{proof}

Combining the results presented in this section and the ones stated in the {section 1} on classifying
topoi, we obtain the following:


\begin{theorem}The presheaf topos ${\rm \bf Sets}^{\mathcal{C}^{\infty}{\rm \bf Rng}_{{\rm fp}}}$ is a classifying topos for $\mathcal{C}^{\infty}-$\-rings, and the universal $\mathcal{C}^{\infty}-$ring $R$ is the $\mathcal{C}^{\infty}-$ring object in ${\rm \bf Sets}^{\mathcal{C}^{\infty}{\rm \bf Rng}_{{\rm fp}}}$ given by ${\mathcal{C}^{\infty}{\rm \bf Rng}_{{\rm fp}}}(\mathcal{C}^{\infty}(\mathbb{R}),-)$ naturally isomorphic to the forgetful functor from $\mathcal{C}^{\infty}{\rm \bf Rng}_{{\rm fp}}$ to ${\rm \bf Sets}$. Thus, for any Grothendieck topos $\mathcal{E}$ there is an equivalence of categories, natural in $\mathcal{E}$:

$$\begin{array}{cccc}
     & {\rm Geom}\,(\mathcal{E},{\rm \bf Sets}^{\mathcal{C}^{\infty}{\rm \bf Rng}_{{\rm fp}}} ) & \rightarrow & \underline{\mathcal{C}^{\infty}{\rm Rings}}\,(\mathcal{E}) \\
     & f & \mapsto & f^{*}(R)
  \end{array}$$
\end{theorem}




\section{A Classifying Topos for the Theory of local\\ $C^{\infty}-$rings}

Now we describe the $\mathcal{C}^{\infty}-$analog of the Zariski site, whose corresponding topos of sheaves  will be the  classifying topos of the theory of the $\mathcal{C}^{\infty}-$local rings.\\

\subsection{The Smooth Zariski Site}

\hspace{0.5cm}In the following we describe the $\mathcal{C}^{\infty}-$analog of the Zariski site, which classifies the theory of the $\mathcal{C}^{\infty}-$local rings.\\

It is known that the topos of sheaves over the Zariski site classifies the theory of (commutative unital) local rings (see, for example, \cite{MakkaiReyes}). We briefly recall its construction.\\

Let $\cal C$ be (some) skeleton of the category of all finitely presented commutative unital rings, ${\rm \bf CRing}_{\rm fp}$. Given a finitely presented commutative unital ring, $A$, we say that a finite family of ring homomorphisms, $\{ f_i: A \to B_i | i \in \{1, \cdots, n\}\}$ is a ``co-coverage'' of $A$ if, and only if there are $a_1, \cdots, a_n \in A$ with $\langle \{a_1, \cdots, a_n \} \rangle = A $ such that for every $i \in \{ 1, \cdots, n\}$, $(B_i, A \stackrel{f_i}{\rightarrow} B_i) \cong (A[{a_i}^{-1}], \eta_{a_i}: A \to A[{a_i}^{-1}])$. The set of all co-covering families of $A$ is denoted by ${\rm coCov}\,(A)$. Naturally, given any isomorphism $\varphi: A \to B$, $\{ A \stackrel{\varphi}{\rightarrow} B \} \in {\rm coCov}\,(A)$, and  for any set of generators of $A$, $\{ a_1, \cdots, a_n\}$, $\{ \eta_{a_i}: A \to A[{a_i}^{-1}] | i \in \{ 1, \cdots, n\} \} \in {\rm coCov}\,(A)$.\\

Passing to the opposite category, $\mathcal{C}^{\rm op}$, we say that a finite set of arrows $\{ f_i: B_i \to A | i \in \{ 1, \cdots, n\}\}$ is a ``covering family for $A$'' if, and only if $\{ {f_i}^{\rm op}: A \to B_i | i \in \{ 1, \cdots, n\}\} \in {\rm coCov}(A)$, and we write $\{ f_i: B_i \to A | i \in \{ 1, \cdots, n\}\} \in {\rm Cov}\,(A)$. The Grothendieck-Zariski topology on $\mathcal{C}^{\rm op}$ is the one generated by ${\rm Cov}$, $J_{{\rm Cov}}$, that is, given any commutative unital ring $A$, $J_{{\rm Cov}}(A)$  consists of all sieves $S$ on $A$   generated by ${\rm Cov}\,(A)$, that is, $S \subseteq \cup_{C \in {\rm Obj}\,({\cal C})}{\rm Hom}_{{\cal C}}(C,A)$ such that every $g \in S$ factors through some element of ${\rm Cov}\,(A)$.\\

The pair $({\cal C}^{\rm op}, J_{\rm Cov})$ thus obtained is the so-called ``Zariski site''. The topos of sheaves over $({\cal C}^{\rm op}, J_{{\rm Cov}})$, ${\cal Z} = {\rm Sh}\,({\cal C}^{\rm op}, J_{{\rm Cov}})$ is the classifying topos for the theory of local commutative unital rings.\\

In order to define the covering families for $\mathcal{C}^{\infty}-$rings we need, just as in the algebraic case, an appropriate notion of ``a $\mathcal{C}^{\infty}-$ring of fractions'':

\begin{definition}\label{nacional}
Let $A$ be a $\mathcal{C}^{\infty}-$ring and let $S \subseteq A$. The $\mathcal{C}^{\infty}-$ring of
fractions of $A$ with respect to $S$ is a pair $(A\{ S^{-1}\}, \eta_S: A \to A\{ S^{-1}\})$
where $A\{ S^{-1}\}$ is a $\mathcal{C}^{\infty}-$ring and $\eta_S: A \to A\{ S^{-1}\}$ is a $\mathcal{C}^{\infty}-$homomorphism such that:
\begin{itemize}
    \item[{\rm (i)}]{$ \eta_S[S] \subseteq (A\{ S^{-1}\})^{\times}$;}
    \item[{\rm (ii)}]{$\eta_S$ satisfies the following universal property: given any $\mathcal{C}^{\infty}-$homomorphism $g: A \to C$ such that $g[S] \subseteq C^{\times}$, there is a unique $\mathcal{C}^{\infty}-$homomorphism $\widetilde{g}: A\{ S^{-1}\} \to C$ such that the following diagram commutes:

$$\xymatrixcolsep{5pc}\xymatrix{
A \ar[r]^{\eta_S} \ar[dr]_{g} & A\{ S^{-1}\} \ar[d]^{\widetilde{g}}\\
  & C}$$}
\end{itemize}
\end{definition}

As the matter of fact, for each $\mathcal{C}^{\infty}-$ring $A$ and each subset $S \subseteq A$, there exists a $\mathcal{C}^{\infty}-$ring of fractions of $A$ with respect to $S$: 
$$A\{ S^{-1}\} := \dfrac{A\{ x_s | s \in S\}}{I_S},$$

where $I_S = \langle \{ \iota_A(s)\cdot x_s - 1 | s \in S\} \rangle$ and
$$\eta_S := q_{I_S}\circ \iota_A: A \to \dfrac{A\{ x_s | s \in S\}}{\langle \{ \iota_A(s)\cdot x_s - 1 | s \in S\} \rangle}$$

In the \textbf{Theorem 1.4} of \cite{rings1}, I. Moerdijk and G. Reyes give two conditions which capture the notion of ``the $\mathcal{C}^{\infty}-$ring of fractions with respect to one element, $S=\{ a\}, a \in A$''. The following proposition presents its natural extension to arbitrary subsets.

\begin{proposition}(cf. \cite{tese})\label{nacional-pr}Let $A$ be a $\mathcal{C}^{\infty}-$ring and let $S \subseteq A$. Then $\mathcal{C}^{\infty}-$ring of
fractions of $A$ with respect to $S$ is the unique (up to isomorphism)  pair $(B, h)$
where $B$ is a $\mathcal{C}^{\infty}-$ring and $h: A \to B$ is a $\mathcal{C}^{\infty}-$homomorphism such that $h[S] \subseteq B^\times$ satisfying the following conditions:
\begin{itemize}
    \item[{\rm (i)}]{$(\forall \beta \in B)(\exists c \in A)(\exists d \in A)((h(c)\in B^{\times})\& (\beta \cdot h(c) = h(d)))$;}
    \item[{\rm (ii)}]{$(\forall a \in A)(h(a)=0 \to (\exists c \in A)(h(c)\in B^{\times})(a \cdot c = 0))$.}
\end{itemize}
\end{proposition}


We introduce the $\mathcal{C}^{\infty}-$analog of the (algebraic) concept of saturation of a
multiplicative subset of a ring in the following:

\begin{definition}Let $A$ be a $\mathcal{C}^{\infty}-$ring and let $S \subseteq A$. The \textbf{$\mathcal{C}^{\infty}-$saturation of $S$} is given by:
$$S^{\infty-{\rm sat}} = \eta_S^{\dashv}[A\{ S^{-1}\}^{\times}]$$
where $A\{ S^{-1}\}$ and $\eta_S: A \to A\{ S^{-1}\}$ were given in \textbf{Definition \ref{nacional}}
\end{definition}

\textbf{Notation:} In virtue of \textbf{Proposition \ref{nacional-pr}}, given any $\beta \in A\{ S^{-1}\}$, there are $b \in A$ and $c \in S^{\infty-{\rm sat}}$ such that $\beta \cdot \eta_S(c) = \eta_S(d)$, so we write $\beta = \frac{\eta_S(d)}{\eta_S(c)}$. For typographical reasons, whenever $S = \{ a\} \subseteq A$, we also write $A_a$ to denote $A\{ a^{-1}\}$.\\

Combining these concepts, we are able to describe the co-covering families of the smooth Zariski Grothendieck (pre)topology.\\

Let $\mathcal{C}$ be (some) skeleton of  $ {\mathcal{C}^{\infty}{\rm \bf Rng}_{\rm fp}}$. We first define the \textbf{smooth Groth\-endieck-Zariski pretopology} on $\mathcal{C}^{\rm op}$.\\

\textbf{Convention:} We say that a covering family of $A$, $\{ g_j: B_j \to A | j \in J \} \in {\rm Cov}\,(A)$ (or a co-covering family of ${\rm coCov}\,(A)$) \textbf{is generated by a family of $\mathcal{C}^{\infty}-$homomorphisms $\mathcal{F} = \{ f_i: A_i \to A | i \in I \}$} if, and only if $\{ g_j: B_j \to A | j \in J \}$ consists of all the $\mathcal{C}^{\infty}-$homomorphism with codomain $A$ which are isomorphic (in the comma category $\mathcal{C}^{\infty}{\rm \bf Rng}_{\rm fp}\downarrow A$) to some element of $\mathcal{F}$. We shall denote it by:
$$\{ g_j: B_j \to A | j \in J \} \stackrel{\cdot}{=} \langle \{ f_i: A_i \to A | i \in I \}\rangle = \langle \mathcal{F} \rangle$$

The covering families, in our case, will be ``generated'' by the dual (opposite) of the co-covering families defined as follows:\\

Let:

$$\begin{array}{cccc}
{\rm coCov}: & {\rm Obj}\,(\mathcal{C}^{\infty}{\rm \bf Rng}_{\rm fp}) & \rightarrow & \wp(\wp({\rm Mor}(\mathcal{C}^{\infty}{\rm \bf Rng}_{\rm fp}))) \\
& A & \mapsto & {\rm coCov}\,(A)
\end{array}$$

For every  $n-$tuple of elements of $A$, $(a_1, \cdots, a_n) \in A \times A \times \cdots, \times A$, $n \in \mathbb{N}$, such that $\langle a_1, a_2, \cdots, a_n\rangle = A$, a family of $\mathcal{C}^{\infty}-$homomorphisms $k_i : A \to B_i$ such that:

\begin{itemize}
\item[(i)]{For every $i \in \{1, \cdots,n \}$, $k_i(a_i) \in {B_i}^{\times}$;}
\item[(ii)]{For every $i \in \{ 1, \cdots, n\}$, if $k_i(a) = 0$ for some $a \in A$, there is some $s_i \in \{ a_i\}^{\infty-{\rm sat}}$ such that $a \cdot s_i = 0$;}
\item[(iii)]{For every $b \in B_i$ there are $c \in \{ a_i\}^{\infty-{\rm sat}}$ and $d \in A$ such that $b \cdot k_i(c) = k_i(d)$.}
\end{itemize}

will be a co-covering family of the $\mathcal{C}^{\infty}-$ring $A$, that is:

\begin{multline*}{\rm coCov}\,(A) = \{ \mathcal{F} \subseteq \cup_{B \in {\rm Obj}\,(\mathcal{C})}{\rm Hom}_{\mathcal{C}^{\infty}{\rm \bf Rng}_{\rm fp}}(A,B)|\mathcal{F} = \\
= \{ k_i: A \rightarrow B_i | (n \in \mathbb{N})\& (i \in \{1,\cdots,n\}) \& \,\, k_i \,\, \mbox{satisfies}\,\, {\rm (i)}, \,\, {\rm (ii)}\,\, \mbox{and}\,\, {\rm (iii)}\} \}
\end{multline*}

In other words,

\begin{multline*}{\rm coCov}\,(A) = \\
= \{ \mathcal{F} \subseteq \cup_{B \in {\rm Obj}\,(\mathcal{C})}{\rm Hom}_{\mathcal{C}^{\infty}-{\rm \bf Rng}_{\rm fp}}| \mathcal{F} = \langle \eta_{a_i}: A \to A\{{a_i}^{-1}\} | i \in \{ 1, \cdots, n\} \rangle\}
\end{multline*}

In terms of diagrams, the ``generators'' of the co-covering families are\\ given by the following arrows:

$$\xymatrix{
A\{ {a_1}^{-1}\} & A\{ {a_2}^{-1}\} & \cdots & A\{ a_{n-1}^{-1}\} & A\{ {a_n}^{-1}\}\\
 & & A \ar@/^/[ull]^{\eta_{a_1}} \ar[ul]_{\eta_{a_2}} \ar[ur]^{\eta_{a_{n-1}}} \ar@/_/[urr]_{\eta_{a_n}}& &
}$$

Given a finitely presented $\mathcal{C}^{\infty}-$ring, a \textbf{covering family} for $A$ in $\mathcal{C}^{\infty}-{\rm \bf Rng}_{{\rm fp}}^{{\rm op}}$ is given by:

$${\rm Cov}\,(A) = \{f^{{\rm op}}: B \to A | (f: A \to B) \in {\rm coCov}\,(A)\}$$

The following technical results are needed in the sequel:

\begin{lemma}\label{imp} Let $A$ be a $\mathcal{C}^{\infty}-$ring and let $a \in A$ and $\beta \in A\{a^{-1}\}$. Since $\beta = \eta^A_a(b)/\eta^A_a(c)$ for some $b \in A$ and $c \in \{a\}^{\infty-sat}$, then there is a unique $\mathcal{C}^{\infty}-$isomorphism of $A$-algebras:

$$\theta_{ab}: (A\{a^{-1}\})\{\beta^{-1}\} \stackrel{\cong}{\longrightarrow} A\{(a \cdot b)^{-1}\}$$

I.e., $\theta_{ab}: (A\{a^{-1}\})\{\beta^{-1}\} \to A\{(a\cdot b)^{-1}\}$ is a $\mathcal{C}^{\infty}-$rings isomorphism such that the following diagram commutes:

$$\xymatrixcolsep{5pc}\xymatrix{
A \ar[r]^{\eta^A_a} \ar@/_/[drr]_{\eta^A_{a \cdot b}} & A\{a^{-1}\}  \ar[r]^{\eta^{A_a}_{\beta}} & (A\{a^{-1}\})\{\beta^{-1}\} \ar[d]^{\theta_{ab}}\\
  & & A\{(a\cdot b)^{-1}\}
 }$$

that is, $(\eta^A_{a \cdot b}: A \to A\{(a \cdot b)^{-1}\}) \cong (\eta^{A_a}_\beta\circ \eta^A_a : A \to (A\{a^{-1}\})\{\beta^{-1}\})$ in $A \downarrow \mathcal{C}^{\infty}{\rm \bf Rng}_{\rm fp}$. Hence:

$$\langle \{ \eta_{a \cdot b}: A \to A\{(a \cdot b)^{-1}\}|a,b \in A \}\rangle = \langle \{\eta_{\beta}\circ \eta_a : A \to (A\{a^{-1}\})\{\beta^{-1}\}| a,b \in A \}\rangle.$$
\end{lemma}

\begin{lemma}\label{par}Let $A$ and $B$ be two $\mathcal{C}^{\infty}-$rings and $S \subseteq A$ and $f: A \to B$ a $\mathcal{C}^{\infty}-$homomorphism. By the universal property of $\eta_S : A \to A\{ S^{-1}\}$ we have a unique $\mathcal{C}^{\infty}-$homomorphism $f_S : A\{ S^{-1}\} \to B\{ f[S]^{-1}\}$ such that the following square commutes:
$$\xymatrixcolsep{5pc}\xymatrix{
A \ar[r]^{\eta_S} \ar[d]_{f} & A\{ S^{-1}\} \ar[d]^{\exists ! f_S}\\
B \ar[r]_{\eta_{f[S]}} & B\{ f[S]^{-1}\}
}.$$

The diagram:
$$\xymatrixcolsep{5pc}\xymatrix{
B \ar[dr]^{\eta_{f[S]}}& \\
 & B \{ f[S]^{-1}\}\\
A\{ S^{-1}\} \ar[ur]_{f_{S}}
}$$
is a pushout of the diagram:
$$\xymatrixcolsep{5pc}\xymatrix{
  & B\\
A \ar[ur]^{f} \ar[dr]_{\eta_S} & \\
   & A\{ S^{-1}\}
}$$
\end{lemma}

\begin{remark}
Note that if $A$ is a finitely presented $\mathcal{C}^{\infty}$-ring, i.e. $A \cong \dfrac{\mathcal{C}^{\infty}(\mathbb{R}^n)}{\langle f_1, \cdots, f_k\rangle}$,  and $b \in A$, then $A\{b^{-1}\}$ is a finitely presented $\mathcal{C}^{\infty}$-ring:
$$A\{b^{-1}\} \cong \dfrac{A\{x\}}{(bx-1)} \cong \dfrac{\mathcal{C}^{\infty}(\mathbb{R}^{n+1})}{\langle f_1\circ \pi_1, \cdots, f_k\circ\pi_1, (bx-1) \circ \pi_2\rangle}$$

\end{remark}

\begin{remark}
Note that if $A$ is a finitely presented $\mathcal{C}^{\infty}$-ring, i.e. $A \cong \dfrac{\mathcal{C}^{\infty}(\mathbb{R}^n)}{\langle f_1, \cdots, f_k\rangle}$,  and $b \in A$, then $A\{b^{-1}\}$ is a finitely presented $\mathcal{C}^{\infty}$-ring:
$$A\{b^{-1}\} \cong \dfrac{A\{x\}}{(bx-1)} \cong \dfrac{\mathcal{C}^{\infty}(\mathbb{R}^{n+1})}{\langle f_1\circ \pi_1, \cdots, f_k\circ\pi_1, (bx-1) \circ \pi_2\rangle}$$
\end{remark}

\begin{definition}Let $A$ be a  $\mathcal{C}^{\infty}-$ring and let $I \subseteq A$ be an ideal. The \textbf{$\mathcal{C}^{\infty}-$radical ideal of $I$} is given by:
$$\sqrt[\infty]{I}=\left\{ a \in A | \left( \frac{A}{I}\right)\{ {a+I}^{-1}\} \cong 0 \right\}$$
\end{definition}

\begin{definition}Given a $\mathcal{C}^{\infty}-$ring $A$, the smooth Zariski spectrum of $A$ is given by the set:
$${\rm Spec}^{\infty}(A)=\{ \mathfrak{p} \in {\rm Spec}\,(A) | \sqrt[\infty]{\mathfrak{p}}=\mathfrak{p} \}$$
together with the topology generated by the following sub-basic sets:
$$D^{\infty}(a)=\{ \mathfrak{p} \in {\rm Spec}^{\infty}(A)| a \notin \mathfrak{p}\}$$
for each $a \in A$.
\end{definition}

\begin{remark}\label{str-shf}Given a $\mathcal{C}^{\infty}-$ring, $A$, we can form a $\mathcal{C}^{\infty}-$(locally) ringed space, $({\rm Spec}^{\infty}(A), \Sigma_A)$, where:
$$\Sigma_A: {\rm Open}({\rm Spec}^{\infty}(A), \subseteq)^{\rm op} \rightarrow \mathcal{C}^{\infty}{\rm \bf Rng}$$
is the (essentially) unique presheaf such that for every $a \in A$ we have:
$$\Sigma_A(D^{\infty}(a))\cong A\{ a^{-1}\}.$$

As proved in \textbf{Proposition 1.6} of \cite{rings2}, $\Sigma_A$ is a sheaf of $\mathcal{C}^{\infty}-$rings whose stalks are local $\mathcal{C}^{\infty}-$rings. More precisely, for each $\mathfrak{p} \in {\rm Spec}^{\infty}(A)$,
$$A_{\mathfrak{p}} = \varprojlim_{a \notin \mathfrak{p}} A\{ a^{-1}\} \cong A\{ {A \setminus \mathfrak{p}}^{-1}\}.$$
\end{remark}

\begin{proposition}(cf. \cite{tese})Let $A$ be a $\mathcal{C}^{\infty}-$ring. A family $\{a_1, \cdots, a_n \} \subseteq A$ is such that $\langle \{a_1, \cdots, a_n \} \rangle = A$ if, and only if:
$${\rm Spec}^{\infty}(A) = \bigcup_{i=1}^{n}D^{\infty}(a_i)$$

\end{proposition}

\begin{proposition} ${\rm Cov}$ is a Grothendieck pretopology on $\mathcal{C}^{\infty}{\rm \bf Rng}_{{\rm fp}}^{{\rm op}}$.

\end{proposition}
\begin{proof} Let $A$ be any finitely presented $\mathcal{C}^{\infty}-$ring.

{\bf Isomorphism axiom:}\\
 Whenever $\varphi^{{\rm op}} : A' \to A$ is a $\mathcal{C}^{\infty}-$isomorphism, the family $\{ \varphi^{{\rm op}}: A' \to A\} \in {\rm Cov}\,(A)$.\\

Note that $\varphi^{{\rm op}}: A' \to A$ is a $\mathcal{C}^{\infty}-$isomorphism in $\mathcal{C}^{\infty}{\rm \bf Rng}_{{\rm fp}}^{{\rm op}}$ if, and only if $\varphi: A \to A'$ is a $\mathcal{C}^{\infty}-$isomorphism in $\mathcal{C}^{\infty}{\rm \bf Rng}_{{\rm fp}}$. Thus, we are going to show that if $\varphi: A \to A'$ is a $\mathcal{C}^{\infty}-$isomorphism, then $\{ \varphi: A \to A' \} \in {\rm coCov}\,(A)$. \\

Indeed, $1_A \in A$ is such that $\langle 1_A \rangle = A$, so the one element family  $\{ \eta_{1_A}: A \rightarrow A\{{1_A}^{-1}\} \}  \in {\rm coCov}\,(A)$.\\

Since $\varphi: A \to A'$ is a $\mathcal{C}^{\infty}-$isomorphism, $\varphi$ is, in particular, a $\mathcal{C}^{\infty}-$homomor\-phism, and we have, for every $s \in \{ 1_A\}^{\infty-{\rm sat}} = A^{\times}$, $\varphi(s) \in {A'}^{\times}$. Also, if $\varphi(a)=0_{A'}$ for some $a \in A$, since $\ker \varphi = \{ 0_A\}$ (for $\varphi$ is injective), $a = 0_A$, so for every $s_i \in \{ 1_{A'}\}^{{\rm \infty-sat}}$ (in particular, there is some such $s_i$) one has $a \cdot s_i = 0_A$. \\

Finally, given any $a' \in A'$, since $\varphi$ is surjective, there is some element $a \in A$ such that $\varphi(a) = a'$. Since $1_A \in \{ 1_{A'}\}^{\infty-{\rm sat}}$ and $a = \frac{a}{1_A}$, we have:
$$a' = \varphi(a)\cdot \varphi(1_A)^{-1}.$$

Since $\varphi: A \to A'$ satisfies (i), (ii) and (iii), the one-element family $\{ \varphi: A \to A'\}$ co-covers $A$, so $\{ \varphi^{{\rm op}}: A' \to A\} \in {\rm Cov}\,(A)$.\\

{\bf Stability axiom:}\\
Now we are going to show that our definition of ${\rm Cov}$ is stable under pullbacks, that is:\\
 If $(a_1, \cdots, a_n) \in A \times A \times \cdots \times A$ is a $n-$tuple such that $\langle a_1, \cdots, a_n \rangle = A$ and $\{ \eta_{a_i}^{\rm op}: A\{{a_i}^{-1}\}  \rightarrow A \}_{i = 1, \cdots, n}$ generates a covering family for $A$, then given a $\mathcal{C}^{\infty}-$rings homomorphism $g: A \to B$, since $\eta_{g(a_i)}\circ g$ is such that $(\eta_{g(a_i)}\circ g)(a_i) \in B\{g(a_i)^{-1}\}^{\times}$, by the universal property of $\eta_{a_i}: A \to A\{{a_i}^{-1}\}$ there is a unique $\mathcal{C}^{\infty}-$homomorphism:

$$g': A\{{a_i}^{-1}\} \to B\{g(a_i)^{-1}\}$$

such that the following diagram commutes:

$$\xymatrixcolsep{5pc}\xymatrix{
A \ar[r]^{\eta_{a_i}} \ar[d]_{g} & A\{{a_i}^{-1}\} \ar[d]^{g'}\\
B \ar[r]_{\eta_{g(a_i)}} & B\{g(a_i)^{-1}\}
 }$$

By \textbf{Lemma \ref{par}}, the diagram above is a pushout, so

$$\xymatrixcolsep{5pc}\xymatrix{
B\{g(a_i)^{-1}\} \ar[d]_{{g'}^{{\rm op}}} \ar[r]^{(\eta_{g(a_i)})^{\rm op}} & B \ar[d]^{{g}^{{\rm op}}}\\
A\{{a_i}^{-1}\} \ar[r]_{{(\eta_{a_i})}^{{\rm op}}} & A
}$$

is a pullback in $\mathcal{C}^{\infty}{\rm \bf Rng}_{{\rm fp}}^{{\rm op}}$.\\

In order to show that the family $\{ (\eta_{g(a_i)})^{\rm op}: B\{g(a_i)^{-1}\} \to B | i=1, \cdots, n\}$ belongs to ${\rm Cov}\,(B)$, it suffices to show that $\{\eta_{g(a_i)}: B \to B\{g(a_i)^{-1}\} | i=1, \cdots, n \}$ belongs to ${\rm coCov}\,(B)$.\\

Since $a_1, \cdots, a_n$ are such that $\langle a_1, \cdots, a_n\rangle = A$, there are some $\lambda_1, \cdots, \lambda_n \in A$ such that:

$$1_A = \sum_{i=1}^{n}\lambda_i \cdot a_i$$

Since $g: A \to B$ is a $\mathcal{C}^{\infty}-$homomorphism, we have:

$$1_B = g(1_A) = \sum_{i=1}^{n} g(\lambda_i)\cdot g(a_i),$$

thus $\langle g(a_1), \cdots, g(a_n)\rangle = B$. Also, since for every $i=1, \cdots, n$, $\eta_{g(a_i)}: B \to B\{g(a_i)^{-1}\}$ is a $\mathcal{C}^{\infty}-$ring of fractions, it follows that $\{\eta_{g(a_i)}: B \to B\{g(a_i)^{-1}\} | i=1, \cdots, n \} \in {\rm coCov}\,(B)$, hence:

$$\{ (\eta_{g(a_i)})^{\rm op}: B\{g(a_i)^{-1}\} \to B | i=1, \cdots, n\} \in {\rm Cov}\,(B).$$

{\bf Transitivity axiom:} \\
If $\{ \eta^{A}_{a_i}: A \to A\{{a_i}^{-1}\} | i=1, \cdots, n \}$ generates a co-covering family of $A$ and for each $i$, $\{\eta^{A_i}_{\beta_{ij}}: A\{{a_i}^{-1}\} \to (A\{{a_i}^{-1}\})\{\beta_{ij}^{-1}\} | j \in \{ 1, \cdots, n_i\}\}$ generates a co-covering family of $A\{{a_i}^{-1}\}$, then:\\

$$\{ \eta^{A_{a_i}}_{\beta_{ij}}\circ \eta^A_{a_i}: A \rightarrow (A\{{a_i}^{-1}\})\{\beta_{ij}^{-1}\} | i \in \{ 1, \cdots, n\} \& j \in \{1, \cdots, n_i \}\}$$

generates a co-covering family of $A$.\\

To show that the transitive axiom holds we will need the following technical result on ``Smooth Commutative Algebra'':


 If for each $i \leq n$ and each $\beta_{ij} \in A\{a_i^{-1}\}$, $j \leq n_i$, we write $\beta_{ij} = \eta_a(b_{ij})/\eta_a(c_{ij})$, with $c_{ij} \in \{a_i\}^{\infty-{\rm sat}}$, then by \textbf{Lemma \ref{imp}}, to show that:



$$\{ \eta^{A_{a_i}}_{\beta_{ij}}\circ \eta^A_{a_i}: A \rightarrow (A\{{a_i}^{-1}\})\{\beta_{ij}^{-1}\} | i \in \{ 1, \cdots, n\} \& j \in \{1, \cdots, n_i \}\}$$

generates a co-covering family of  $A$ amounts to show that:

$$\{ \eta^A_{a_i \cdot b_{ij}} : A \to A\{({a_i}\cdot {b_{ij}})^{-1}\} | (i \in \{ 1, \cdots n\}) \& (j \in \{1, \cdots, n_i\})\}$$

does.\\

By hypothesis, $\{ \eta_{a_i}^{\rm op}: A\{ {a_i}^{-1}\} \to A | i \in \{ 1, \cdots, n\}\}\}$ generates a covering family of $A$, so:

$$1_A \in \langle a_1, \cdots, a_n\rangle$$

or, \underline{equivalently}:

$${\rm Spec}^{\infty}\,(A) = \bigcup_{i=1}^{n}D_A^{\infty}(a_i)$$

Since for every $i \in \{ 1, \cdots, n\}$ we have a canonical homeomorphism:

$$\begin{array}{cccc}
    \varphi: & {\rm Spec}^{\infty}\,(A\{ {a_i}^{-1}\}) & \rightarrow & D^{\infty}\,(a_i) \\
     & \mathfrak{p} & \mapsto & \eta_{a_i}^{-1}[\mathfrak{p}]
  \end{array}$$

Also by hypothesis, for any $i \in \{ 1, \cdots, n\}$, $\{ (\eta^A_{a_ib_{ij}})^{\rm op}: A\{ {a_i \cdot b_{ij}}^{-1}\} \rightarrow A\{ {a_i}^{-1}\} | j \in \{ 1, \cdots, n_i\}  \}$ generates a covering family of $A\{ {a_i}^{-1}\}$, so:

$${\rm Spec}^{\infty}(A\{ {a_i}^{-1}\}) = \bigcup_{j=1}^{n_i} D_{A\{a_i^-1\}}^{\infty}(\beta_{ij}) \approx \bigcup_{j=1}^{n_i} D_A^{\infty}(a_i) \cap  D_A^{\infty}(b_{ij}) \approx \bigcup_{j=1}^{n_i} D_A^{\infty}(a_i \cdot b_{ij})$$

Putting all together we obtain:

$${\rm Spec}^{\infty}\,(A) = \bigcup_{i=1}^{n} D_A^{\infty}\,(a_i) \approx \bigcup_{i=1}^{n} {\rm Spec}^{\infty}\,(A\{ {a_i}^{-1}\}) \approx \bigcup_{i=1}^{n}\left( \bigcup_{j=1}^{n_i} D_A^{\infty}\,(a_i \cdot b_{ij})\right)$$

thus,

$${\rm Spec}^{\infty}\,(A) = \bigcup_{\substack{i \leq n\\ j \leq n_i}} D_A^{\infty}\,(a_i \cdot b_{ij})$$

but this is equivalent to

$$1_A \in \langle \{a_i \cdot b_{ij}: i \leq n, j \leq n_i\} \rangle$$

and the transitivity is proved.\\

Thus, ${\rm Cov}$ defines a Grothendieck pretopology on $\mathcal{C}^{\infty}{\rm \bf Rng}_{\rm fp}^{\rm op}$. We have:

$$J_{{\rm Cov}}: {\rm Obj}\,(\mathcal{C}^{\infty}{\rm \bf Rng}_{\rm fp}) \to \wp(\wp({\rm Mor}\,( \mathcal{C}^{\infty}{\rm \bf Rng}_{\rm fp})))$$

given by:

$$J_{{\rm Cov}}\,(A) := \{ \overleftarrow{S} \subseteq \cup_{B \in {\rm Obj}\,(\mathcal{C}^{\infty}{\rm \bf Rng}_{\rm fp})} {\rm Hom}_{\mathcal{C}^{\infty}{\rm \bf Rng}_{\rm fp}}\,(B,A) | S \in {\rm Cov}\,(A) \}$$

turning $(\mathcal{C}^{\infty}{\rm \bf Rng}_{\rm fp}^{\rm op}, J_{{\rm Cov}})$ into a small site - the so called smooth Zariski site.\\
\end{proof}

\begin{proposition} \label{sheaf-struc} Let $I$ be a finite set (say $I = \{1, \cdots, n\}$) and let
$\{ A\{a_i^{-1}\} \overset{\eta_{a_i}}\to A | i \in I\}$ be a
$J_{{\rm Cov}}$-covering of $A$ in ${\mathcal{C}^{\infty}{\rm \bf Rng}_{\rm
fp}}^{\rm op}$, then the diagram
below is an equalizer in the category of $\mathcal{C}^{\infty}$-rings:

(E) $$ A \to \prod_{i \in I} A\{a_i^{-1}\} \rightrightarrows
\prod_{i,j \in I} A\{(a_i.a_j)^{-1}\}$$
\end{proposition}

\begin{proof}
By hypothesis, $A = \langle \{a_i | i \in I\}\rangle$, or
equivalently, ${\rm Spec}^{\infty}(A) = D^\infty(1) = \bigcup_{i \in I}
D^\infty(a_i)$.

Since the affine $\mathcal{C}^{\infty}$-(locally) ringed space of $A$, $({\rm Spec}^{\infty}(A),\Sigma_A)$, given in \textbf{Remark \ref{str-shf}},
is in particular a sheaf of $\mathcal{C}^{\infty}$-rings, then the diagram below is an equalizer in the category of $\mathcal{C}^{\infty}$-rings:

$$ \xymatrix{\Sigma_A(D^\infty(1)) \ar[r]& \displaystyle\prod_{i \in I} \Sigma_A(D^\infty(a_i))
\ar@<1ex>[r] \ar@<-1ex>[r] & \displaystyle\prod_{i,j \in I} \Sigma_A(D^\infty(a_i)
\cap D^\infty(a_j))}$$

As $D^\infty(a_i) \cap D^\infty(a_j) = D^\infty(a_i.a_j)$ and
$\Sigma_A(D^\infty(b)) \cong A\{b^{-1}\}$, we have that the diagram
of $\mathcal{C}^{\infty}$-rings below is an equalizer:

$$A \overset{\cong}\to A\{1^{-1}\} \to \prod_{i \in I}
A\{a_i^{-1}\} \rightrightarrows \prod_{i,j \in I} A\{(a_i \cdot a_j)\}^{-1}$$

and this finishes the proof.
\end{proof}

We define the \textbf{smooth Grothendieck-Zariski topos}, that we denote by ${\cal Z}^{\infty}$, as the topos of sheaves over the smooth Zariski site:

$${\cal Z}^{\infty} = {\rm Sh}\,(\mathcal{C}^{\rm op}, J_{{\rm Cov}}),$$

where $\mathcal{C}$ is a skeleton of the category of all finitely presented $\mathcal{C}^{\infty}-$rings, $\mathcal{C}^{\infty}{\rm \bf Rng}_{\rm f.p.}$.


\begin{remark} \label{struct-re}
The forgetfull functor $\mathcal{O} : \mathcal{C}^{\infty}{\rm \bf Rng}_{\rm fp} \to {\rm \bf Sets}$ is
called the
\underline{structure sheaf} of the Grothendieck-Zariski smooth topos.

This is actually a sheaf of sets since if $\{\eta_{a_i} : A \to
A\{a_i^{-1}\} | i \leq n\}$ is a smooth Zariski co-covering
(i.e. $A = \langle \{a_1, \cdots, a_n\} \rangle$), then the diagram of sets
below must be an equalizer,
$$ A \to \prod_{i \in I} A\{a_i^{-1}\} \rightrightarrows \prod_{i,j \in
I} A\{(a_i.a_j)^{-1}\}$$
since it is indeed an equalizer of $\mathcal{C}^{\infty}$-rings and the forgetfull functor
$\mathcal{C}^{\infty}{\rm \bf Rng} \to {\rm \bf Sets}$ preserves limits.
\end{remark}

\begin{proposition}\label{papel}The following rectangle is a pushout:\\

$$\xymatrixcolsep{5pc}\xymatrix{
A  \ar@{-->}[dr]^{\eta_{a_i \cdot a_j}} \ar[r]^{\eta_{a_i}} \ar[d]_{\eta_{a_j}} & A\{ {a_i}^{-1}\} \ar[d]\\
A\{ {a_j}^{-1}\} \ar[r] & A\{(a_i \cdot a_j)^{-1}\}
}$$
\end{proposition}

\begin{theorem}\label{subcan} The smooth Grothendieck-Zariski topology $J_{\rm Cov}$ on ${\cal Z}^{\infty}$ is
subcanonical, that is, for every finitely presented $\mathcal{C}^{\infty}-$ring
$B$, the representable functor:

$$\begin{array}{cccc}
{\rm Hom}_{\cal C}\,(-, B) : & {\cal C}^{op} & \rightarrow
& {\rm \bf Set}\\
& A & \mapsto & {\rm Hom}_{\cal C}\,(A,B)\\
&(\xymatrix{A_1 \ar[r]^f & A_2}) & \mapsto & (\xymatrix{{\rm
Hom}_{{\cal C}}\,(A_2,B) \ar[r]^{- \circ f} &
{\rm Hom}_{{\cal C}}\,(A_1,B))}
\end{array}$$
is a sheaf (of sets).
\end{theorem}
\begin{proof}

Let $I$ be a finite set (lets say $I = \{1, \cdots, n\}$) and let
$\{ A\{a_i^{-1}\} \overset{\eta_{a_i}}\to A | i \in I\}$ be a
${\cal Z}^\infty$-covering of $A$ in ${\cal C}$.

Recall, from \textbf{Proposition \ref{papel}}, that for every $i,j \in I$, the following rectangle is a pushout in $\mathcal{C}^{\infty}{\rm \bf Rng}$:

$$\xymatrixcolsep{5pc}\xymatrix{
A  \ar@{-->}[dr]^{\eta_{a_i \cdot a_j}} \ar[r]^{\eta_{a_i}} \ar[d]_{\eta_{a_j}} & A\{ {a_i}^{-1}\} \ar[d]\\
A\{ {a_j}^{-1}\} \ar[r] & A\{(a_i \cdot a_j)^{-1}\}
}$$

We must prove that:

(I)  $$ {\cal C}(A,B) \to \prod_{i \in I} {\cal C}(A\{a_i^{-1}\},B)
\rightrightarrows
\prod_{i,j \in I} {\cal C}(A\{(a_i.a_j)^{-1}\},B)$$

is an equalizer diagram of sets and functions.

Since ${\cal C} = {\mathcal{C}^{\infty}{\rm \bf Rng}_{\rm fp}}^{\rm op}$, this
amounts to prove that

(II) $${\mathcal{C}^{\infty}{\rm \bf Rng}_{\rm fp}}(B,A) \to \prod_{i \in
I} {\mathcal{C}^{\infty}{\rm \bf Rng}_{\rm fp}}(B,A\{a_i^{-1}\})
\rightrightarrows \prod_{i,j \in I} {\mathcal{C}^{\infty}{\rm \bf Rng}_{\rm
fp}}(B,A\{(a_i.a_j)^{-1}\})$$

is an equalizer diagram of sets and functions.

As Hom functors preserve products, the diagram (II) is isomorphic to

(III) $${\mathcal{C}^{\infty}{\rm \bf Rng}_{\rm fp}}(B,A) \to
{\mathcal{C}^{\infty}{\rm \bf Rng}_{\rm fp}}(B,\prod_{i \in
I}A\{a_i^{-1}\})
\rightrightarrows {\mathcal{C}^{\infty}{\rm \bf Rng}_{\rm
fp}}(B,\prod_{i,j \in I} A\{(a_i.a_j)^{-1}\})$$

But this is an equalizer diagram of sets and functions since the Hom
functor ${\mathcal{C}^{\infty}{\rm \bf Rng}_{\rm fp}}(B, -)$ preserves
equalizers and
the diagram

(E) $$ A \to \prod_{i \in I} A\{a_i^{-1}\} \rightrightarrows
\prod_{i,j \in I} A\{(a_i.a_j)^{-1}\}$$

is an equalizer in the category of $\mathcal{C}^{\infty}$-rings, by {\bf
Proposition \ref{sheaf-struc}}.

Thus the Grothendieck topology $J_{\rm Cov}$ of the smooth Zariski site is subcanonical.
\end{proof}

\vspace{0.5cm}



Now we show that the topos of sheaves on the smooth Zariski site, that we have just described, is  the classifying topos of the theory of the $\mathcal{C}^{\infty}-$local rings.\\

In order to define a ``local $\mathcal{C}^{\infty}-$ring object'' in a topos, we use - as motivation - the Mitchell-B\'{e}nabou language.  We define a local $\mathcal{C}^{\infty}-$ring object in a topos $\mathcal{E}$ as follows: it is a $\mathcal{C}^{\infty}-$ring object $R$ in $\mathcal{E}$ such that the (geometric) formula:

$$(\forall a \in R)((\exists b \in R)(a \cdot b = 1)\vee (\exists b \in R)((1-a)\cdot b = 1))$$

is valid. \\

By definition, this means that the union of the subobjects:

$$\{ a \in R | \exists b \in R (a \cdot b =1)\} \rightarrowtail R,$$
$$\{ a \in R | \exists b \in R ((1-a) \cdot b =1)\} \rightarrowtail R$$
of $R$ is all of $R$. Equivalently, consider the two subobjects of the product $R \times R$ defined by:

\begin{equation}\label{two}
\begin{cases}
U = \{ (a,b) \in R \times R | a \cdot b = 1\} \rightarrowtail R \times R\\
V = \{ (a,b) \in R \times R | (1-a)\cdot b = 1\} \rightarrowtail R \times R
\end{cases}
\end{equation}

The $\mathcal{C}^{\infty}-$ring object $R$ is local if, and only if, the two composites $U \rightarrowtail R \times R \stackrel{\pi_1}{\rightarrow} R$ and $V \rightarrowtail R \times R \stackrel{\pi_1}{\rightarrow} R$ form an epimorphic family in $\mathcal{E}$.\\

In {section  2}, we have observed that there is an equivalence between $\mathcal{C}^{\infty}-$ring objects $R$ in a topos $\mathcal{E}$ and left exact functors, $\mathcal{C}^{\infty}{\rm \bf Rng}_{{\rm fp}}^{{\rm op}} \to \mathcal{E}$. Explicitly, given such a left-exact functor $F$, the corresponding $\mathcal{C}^{\infty}-$ring object $R$ in $\mathcal{E}$ is $F(\mathcal{C}^{\infty}(\mathbb{R}))$. Conversely, given a $\mathcal{C}^{\infty}-$ring $R$ in $\mathcal{E}$, the corresponding functor:
$$\phi_R: \mathcal{C}^{\infty}{\rm \bf Rng}_{{\rm fp}}^{{\rm op}} \to \mathcal{E}$$
sends the finitely presented $\mathcal{C}^{\infty}-$ring $A = \dfrac{\mathcal{C}^{\infty}(\mathbb{R}^n)}{\langle p_1, \cdots, p_k \rangle}$ to the following equalizer in $\mathcal{E}$:

\begin{equation}\label{five}
\phi_R(A) \rightarrowtail \xymatrixcolsep{5pc}\xymatrix{R^{n} \ar[r]^{(p_1, \cdots, p_k)} \ar@<-1ex>[r]_{(0,\cdots, 0)} & R^k}
\end{equation}

This description readily yields the corresponding definition of $\phi_R$ on arrows.\\


The following lemma gives a condition for a $\mathcal{C}^{\infty}-$ring $R$ in a topos $\mathcal{E}$ to be local, phrased in terms of this corresponding functor $\phi_R$.

\begin{lemma}\label{l2}Let $\mathcal{E}$ be a topos, $R$ be a $\mathcal{C}^{\infty}-$ring object in $\mathcal{E}$, and let $\phi_R : \mathcal{C}^{\infty}{\rm \bf Rng}^{{\rm op}} \to \mathcal{E}$ be the corresponding left exact functor. The following are equivalent:
\begin{itemize}
  \item[(i)]{$R$ is a local $\mathcal{C}^{\infty}-$ring in $\mathcal{E}$;}
  \item[(ii)]{$\phi_R$ sends the pair of arrows in the category $\mathcal{C}^{\infty}{\rm \bf Rng}_{{\rm fp}}$:
  $$\xymatrix{
    & \mathcal{C}^{\infty}(\mathbb{R}) \ar[dl] \ar[dr] & \\
    \dfrac{\mathcal{C}^{\infty}(\mathbb{R})\{ x_f\}}{\langle x_f \cdot \iota_{\mathcal{C}^{\infty}(\mathbb{R})}(f) - 1\rangle} & & \dfrac{\mathcal{C}^{\infty}(\mathbb{R})\{ x_{1-f}\}}{\langle x_{1-f} \cdot \iota_{\mathcal{C}^{\infty}(\mathbb{R})}(1-f) - 1\rangle}
  }$$
  to an epimorphic family of two arrows in $\mathcal{E}$; }
  \item[(iii)]{For any finitely presented $\mathcal{C}^{\infty}-$ring $A$ and any elements $a_1, \cdots, a_n$ such that $\langle a_1, \cdots, a_n\rangle = A$, $\phi_R$ sends the family of arrows in $\mathcal{C}^{\infty}{\rm \bf Rng}_{{\rm fp}}$:

      $$\{ A \to A\{ a_i^{-1}\} | i =1, \cdots, n\}$$

      to an epimorphic family $\{ \phi_R(A\{ {a_i}^{-1}\}) \to \phi_R(A) | i =1, \cdots, n \}$ in $\mathcal{E}$.
      }
\end{itemize}
\end{lemma}
\begin{proof}
Ad $(i) \iff (ii)$: This follows immediately from the explicit description of the functor $\phi_R$. Let $A = \dfrac{\mathcal{C}^{\infty}(\mathbb{R})\{ x_f\}}{\langle x_f \cdot \iota_{\mathcal{C}^{\infty}(\mathbb{R})}(f) - 1\rangle}$ and $B= \dfrac{\mathcal{C}^{\infty}(\mathbb{R})\{ x_{1-f}\}}{\langle x_{1-f} \cdot \iota_{\mathcal{C}^{\infty}(\mathbb{R})}(1-f) - 1\rangle}$. Note that:
$$\phi_R \left( A\right) = \{ (a,b) \in R \times R | a \cdot b = 1 \}$$

To wit, $\phi_R$ sends $A = \dfrac{\mathcal{C}^{\infty}(\mathbb{R})\{ x_f\}}{\langle x_f \cdot \iota_{\mathcal{C}^{\infty}(\mathbb{R})}(f) - 1\rangle}$ to the equalizer:

$$\xymatrixcolsep{5pc}\xymatrix{
\phi_R \left( A \right) \,\, \ar@{>->}[r] & R\times R \ar[r]_{1 \circ !} \ar@<1ex>[r]^{x_f \cdot \iota_{\mathcal{C}^{\infty}(\mathbb{R})}(f)} & R
}$$

and $\phi_R$ sends $B$ to:

$$\xymatrixcolsep{8pc}\xymatrix{
\phi_R \left( B \right)\,\, \ar@{>->}[r] & R\times R \ar[r]_{1 \circ !} \ar@<1ex>[r]^{x_{1-f} \cdot \iota_{\mathcal{C}^{\infty}(\mathbb{R})}(1-f)} & R
}$$

The arrow $\mathcal{C}^{\infty}(\mathbb{R}) \to A$ is mapped into the composite $\alpha: \phi_R(A) \to R$ given by $\phi_R(A) \rightarrowtail R \times R \stackrel{\pi_1}{\rightarrow} R$, and the arrow $\mathcal{C}^{\infty}(\mathbb{R}) \to B$, is mapped into the composite $\beta: \phi_R(B) \to R$, given by  $\phi_R(B) \rightarrowtail R \times R \stackrel{\pi_1}{\rightarrow} R$.\\

By the definition of a local $\mathcal{C}^{\infty}-$ring, (i) is equivalent to (ii).\\

Ad (iii) $\Rightarrow$ (ii): is also clear, since (ii) is the special case of (iii) in which $A = \mathcal{C}^{\infty}(\mathbb{R})$ while $n=2$, $a_1 = f$ and $a_2 = 1 - f$.\\

Ad (ii) $\Rightarrow$ (iii): Assume that (ii) holds, and suppose we are given a finitely presented $\mathcal{C}^{\infty}-$ring $A$ and elements $a_1, \cdots, a_n \in A$ with $\sum_{i=1}^{n}a_i = 1$. This result is proved using induction. We are going to prove the:\\

\textbf{Case  $n=2$}.\\

In this case $a_2 = 1 - a_1$. We form the pushouts of $\mathcal{C}^{\infty}(\mathbb{R}) \rightarrow \dfrac{\mathcal{C}^{\infty}(\mathbb{R})\{ x_f\}}{\langle x_f \cdot \iota_{\mathcal{C}^{\infty}(\mathbb{R})}(f) - 1\rangle}$ and
$\mathcal{C}^{\infty}(\mathbb{R}) \rightarrow \dfrac{\mathcal{C}^{\infty}(\mathbb{R})\{ x_{1-f}\}}{\langle x_{1-f} \cdot \iota_{\mathcal{C}^{\infty}(\mathbb{R})}(1-f) - 1\rangle}$ along the map $\mathcal{C}^{\infty}(\mathbb{R}) \rightarrow A$ sending ${\rm id}_{\mathbb{R}}$ to $a_1$, as in:

$$\xymatrixcolsep{5pc}\xymatrix{
\dfrac{\mathcal{C}^{\infty}(\mathbb{R})\{ x_{1-f}\}}{\langle x_{1-f} \cdot \iota_{\mathcal{C}^{\infty}(\mathbb{R})}(1-f) - 1\rangle} \ar[d] & \mathcal{C}^{\infty}(\mathbb{R}) \ar[d]^{a_1} \ar[l] \ar[r] & \dfrac{\mathcal{C}^{\infty}(\mathbb{R})\{ x_f\}}{\langle x_f \cdot \iota_{\mathcal{C}^{\infty}(\mathbb{R})}(f) - 1\rangle} \ar[d]\\
A\{ (1-a_1)^{-1}\} & A \ar[l] \ar[r] & A\{ {a_1}^{-1}\}
}$$

giving the indicated $\mathcal{C}^{\infty}-$rings of fractions $A\{ (1-a_1)^{-1}\}$ or $A\{ {a_1}^{-1}\}$. These squares are pullbacks in $\mathcal{C}^{\infty}{\rm \bf Rng}_{\rm fp}^{\rm op}$, hence they are sent by the left-exact functor $\phi_R$ to pullbacks in $\mathcal{E}$, as in:

$$\xymatrix{
\phi_R\left(\dfrac{\mathcal{C}^{\infty}(\mathbb{R})\{ x_{1-f}\}}{\langle x_{1-f} \cdot \iota_{\mathcal{C}^{\infty}(\mathbb{R})}(1-f) - 1\rangle}\right)\ar[r]^(.8){\pi_1} & R & \ar[l]_(0.7){\pi_1} \phi_R\left( \dfrac{\mathcal{C}^{\infty}(\mathbb{R})\{ x_f\}}{\langle x_f \cdot \iota_{\mathcal{C}^{\infty}(\mathbb{R})}(f) - 1\rangle} \right)\\
\phi_R(A\{ (1-a_1)^{-1}\}) \ar[u] \ar[r] & \phi_R(A) \ar[u]& \phi_R(A\{ {a_1}^{-1}\}) \ar[l] \ar[u]
}$$

But by assumption
$$\phi_R\left( \dfrac{\mathcal{C}^{\infty}(\mathbb{R})\{ x_{1-f}\}}{\langle x_{1-f} \cdot \iota_{\mathcal{C}^{\infty}(\mathbb{R})}(1-f) - 1\rangle}\right) \to R$$ and
$$\phi_R\left(  \dfrac{\mathcal{C}^{\infty}(\mathbb{R})\{ x_f\}}{\langle x_f \cdot \iota_{\mathcal{C}^{\infty}(\mathbb{R})}(f) - 1\rangle} \right) \to R$$
form an epimorphic family in $\mathcal{E}$, and hence so does the pullback of this family. This proves (iii) for the case $n=2$.\\

The general case follows by induction. For instance, if $n =3$ and $a_1+a_2+a_3 = 1$, let $\beta \in A\{(a_2+a_3)^{-1}\}$ such that $\beta. \eta(a_2) + \beta. \eta(a_3) = 1$. Then, again by the case $n=2$, $\phi_R$ sends the three arrows in $\mathcal{C}^{\infty}{\rm \bf Rng}_{\rm fp}$
$$ A \to A\{a_1^{-1}\}$$
$$ A \to A\{a_2^{-1}\} \to  A\{(a_2+a_3)^{-1}\}\{(\beta. \eta(a_2))^{-1}\}$$
$$ A \to A\{a_3^{-1}\} \to  A\{(a_2+a_3)^{-1}\}\{(\beta. \eta(a_3))^{-1}\}$$
to an epimorphic family in ${\cal E}$. Thus $\phi_R$ also sends the family of canonical arrows $\{ A \to A\{a_i^{-1}\} : i =1,2,3\}$ to an epimorphic family in ${\cal E}$.
\end{proof}


\begin{theorem}\label{above}The smooth Grothendieck-Zariski topos $\mathcal{Z}^{\infty} = {\rm Sh}\,(\mathcal{C}^{\rm op}, J_{{\rm Cov}}),$ is a classifying topos for local $\mathcal{C}^{\infty}-$rings, i.e., for any Grothendieck topos $\mathcal{E}$, there is an equivalence of categories:
\begin{equation}
\label{seven}
{\rm Geom}\,(\mathcal{E}, \mathcal{Z}^{\infty}) \simeq \mathcal{C}^{\infty}{\rm LocRng}\,(\mathcal{E})
\end{equation}
where $\mathcal{C}^{\infty}{\rm  LocRng}\,(\mathcal{E})$ is the category of local $\mathcal{C}^{\infty}-$ring-objects in $\mathcal{E}$.\\

The universal local $\mathcal{C}^{\infty}-$ring is the structure sheaf $\mathcal{O}$ of the Grothendieck-Zariski smooth topos (see \textbf{Remark \ref{struct-re}}).
\end{theorem}
\begin{proof}
As a special case of 
the results on classifying topoi presented in the section 1, there is an equivalence between
${\rm Geom}\,(\mathcal{E}, \mathcal{Z}^{\infty})$ and the category of \underline{continuous} left-exact functors from $\mathcal{C}^{\infty}{\rm \bf Rng}_{{\rm fp}}$ to $\mathcal{E}$.\\

This category is equivalent to the full subcategory $\mathcal{C}^{\infty}{\rm  LocRng}\,(\mathcal{E})$ consisting of local $\mathcal{C}^{\infty}-$rings.\\

The identification of the universal local $\mathcal{C}^{\infty}-$ring is the object of $\mathcal{Z}^{\infty}$ represented by the object $\mathcal{C}^{\infty}(\mathbb{R})$ of the Grothendieck Zariski smooth site, this is precisely the structure sheaf (= forgetful functor) $\mathcal{O} : \mathcal{C}^{\infty}{\rm \bf Rng}_{\rm fp} \to {\rm \bf Sets}$.
\end{proof}


\section{A Classifying Topos for the Theory of the von Neumann-regular $\mathcal{C}^{\infty}-$rings}

\hspace{0.5cm}A \textit{von Neumann regular $\mathcal{C}^{\infty}-$ring} is a $\mathcal{C}^{\infty}-$ring $A$ such that one of the (following) equivalent conditions hold:
\begin{itemize}
  \item[(i)]{$(\forall a \in A)(\exists x \in A)(a = a^2x)$;}
  \item[(ii)]{Every principal ideal of $A$ is generated by an idempotent element, \textit{i.e.},
  $$(\forall a \in A)(\exists e \in A)(\exists y \in A)(\exists z \in A)((e^2=e)\& (ey=a) \& (az = e))$$}
  \item[(iii)]{$(\forall a \in A)(\exists ! b \in A)((a=a^2b)\& (b = b^2a))$}
\end{itemize}

For a proof of this result 	in the setting of usual commutative rings, see, for instance, \cite{Mariano}.\\

The class of all von Neumann regular $\mathcal{C}^{\infty}-$rings contains all $\mathcal{C}^{\infty}-$fields and is closed under arbitrary products, quotients and directed limits (cf. \cite{tese}).\\



Any von Neumann regular $\mathcal{C}^{\infty}-$ring $A$ is a $\mathcal{C}^{\infty}-$reduced   $\mathcal{C}^{\infty}-$ring (i.e., $\sqrt[\infty]{(0_A)}=(0_A)$) such that ${\rm Spec}^{\infty}(A)$ is a Boolean space. In an upcoming paper (\cite{BM}) we show that  the converse is also true; moreover, for a fixed $\mathcal{C}^{\infty}-$field, $\mathbb{F}$, we prove that given any Boolean space $(X,\tau)$, there is some $\mathcal{C}^{\infty}-$reduced von Neumann regular $\mathcal{C}^{\infty}-$ring that is an $\mathbb{F}-$algebra, $R_X$, such that ${\rm Spec}^{\infty}(R_X)\approx (X, \tau)$.\\


Now we turn to the problem of constructing a classifying topos for this $\mathcal{C}^{\infty}-$rings. As defined above, a von Neumann-regular $\mathcal{C}^{\infty}-$ring is a $\mathcal{C}^{\infty}-$rings $(A, \Phi)$ in which the first-order formula:

\begin{equation}\label{Vogue}(\forall x \in A)(\exists ! y \in A)((xyx=x)\& (yxy = y))
\end{equation}

holds. Denoting by $\varphi(x,y):= ((xyx=x)\& (yxy=y))$, we note that the formula:
$$(\forall x \in A)(\exists ! y \in A)\varphi(x,y)$$
defines a functional relation from $A$ to $A$, so we can define an unary functional symbol

Let $\mathbb{T}_{\rm vN}$ be the theory of the von Neumann-regular $\mathcal{C}^{\infty}-$rings in the language $\mathcal{L}$ described at the beginning of the first section of this paper. We can define the unary functional symbol $*$ by means of the formula \eqref{Vogue}:

$$\begin{array}{cccc}
    *: & A & \rightarrow & A \\
     & x & \mapsto & y \ {\rm s.t.}\  \varphi(x,y)
  \end{array}$$

in order to obtain a richer language, namely $\mathcal{L}' = \mathcal{L} \cup \{ *\}$.\\

\begin{remark}
Let $\mathbb{T}^{'}$ be the a theory in the language $\mathcal{L}' = \mathcal{L} \cup \{ *\}$, that contains:\\

$\bullet$ the (equational) $\mathcal{L}$-axioms for of $\mathcal{C}^{\infty}-$rings; \\

$\bullet$ the (equational) $\mathcal{L}'$-axiom
$$\sigma:= (\forall x)((xx^{*}x = x)\& (x^*xx^*=x^*))$$

that is, $\mathbb{T}^{'}:= \mathbb{T}\cup \{ \sigma\}$. By the \textbf{Theorem of Extension by Definition} (cf. \textbf{Corollary 4.4.7} of \cite{vanD}), we know that $\mathbb{T}^{'}$ is a conservative extension of $\mathbb{T}$.
\end{remark}

\begin{remark}

(a) Note that in every von Neumann-regular $\mathcal{C}^{\infty}-$ring $V$, since $x^*xx^* = x^*$ holds for every $x \in V$, then $0^* = 0.$

(b) If $\mathbb{F}$ is a $\mathcal{C}^{\infty}-$field, and thus a von Neumann-regular $\mathcal{C}^{\infty}-$ring, we have:
$$\mathbb{F} \models \sigma.$$

Since  $x x^* x = x$ holds for every $x \in \mathbb{F}$, then if $x \neq 0$, we must have

$$x^* = \dfrac{1}{x}.$$

(c) In fact, the unary function $*$ does not belong to the language $\mathcal{L}$.\\

We have seen that $\mathcal{C}^{\infty}(\mathbb{R}^0) \cong \mathbb{R}$, together with its canonical $\mathcal{C}^{\infty}-$structure $\Phi$, is a $\mathcal{C}^{\infty}-$field, thus a von Neumann-regular $\mathcal{C}^{\infty}-$ring, so
$$\mathbb{R} \models \sigma.$$

Now, the function:

$$\begin{array}{cccc}
    *: & \mathbb{R} & \rightarrow & \mathbb{R} \\
     & x & \mapsto & \begin{cases}
                       \dfrac{1}{x}, & \mbox{if } x \neq 0 \\
                       0, & \mbox{otherwise}.
                     \end{cases}
  \end{array}$$

is defined by $\sigma$. However, there is no (continuous, and thus) smooth function $f: \mathbb{R} \to \mathbb{R}$ such that
$$(\forall x \in \mathbb{R})(x^*=f(x)),$$
that is,
$$(\forall f \in \mathcal{C}^{\infty}(\mathbb{R}, \mathbb{R}))(\Phi(f)\neq *)$$
so $*$ is not a symbol of the original language of $\mathcal{C}^{\infty}-$rings.
\end{remark}

\begin{remark}
(a) Since $\varphi(x,y)$ is a conjunction of two equations, the von Neumann-regular $\mathcal{C}^{\infty}-$homomorphisms preserve $*$, \textit{i.e.},
$$(\forall x \in \mathbb{R})(h(x^*)={h(x)}^*)$$

whenever $(A,\Phi)$ and $(B,\Psi)$ are von Neumann-regular $\mathcal{C}^{\infty}-$rings and $h: (A,\Phi) \to (B,\Psi)$
is a von Neumann-regular $\mathcal{C}^{\infty}-$homomorphism.

(b) Since the $\mathcal{L}$-class of von Neumann-regular $\mathcal{C}^{\infty}-$rings is closed under quotients by $\mathcal{C}^{\infty}$-congruences and $\mathcal{C}^{\infty}$-congruences are classified by ideals, it follows from the item (a) that for each  von Neumann-regular $\mathcal{C}^{\infty}-$ring $V$ and any ideal $I \subseteq V$, then $x^* - y^* \in I$ whenever $x -y \in I$.
\end{remark}

\begin{definition}A finitely presented von Neumann regular $\mathcal{C}^{\infty}-$ring is a von Neumann-regular $\mathcal{C}^{\infty}-$ring $(V,\Phi)$ such that there is a finite set $X$ and an ideal $I \subseteq L(X) = {\rm vN}\,(\mathcal{C}^{\infty}(\mathbb{R}^{X}))$ with:
$$V \cong \dfrac{{\rm vN}\,(\mathcal{C}^{\infty}(\mathbb{R}^{X}))}{I}$$
\end{definition}

\begin{remark}$(V,\Phi)$ is a finitely presented von Neumann-regular $\mathcal{C}^{\infty}-$ring if, and only if, the representable functor:
$${\rm Hom}_{\mathcal{C}^{\infty}{\rm \bf {\rm \bf vNRng}}}(V,-): \mathcal{C}^{\infty}{\rm \bf {\rm \bf vNRng}} \to {\rm \bf Set}$$
preserves all directed colimits. That is to say that for every directed system of von Neumann-regular $\mathcal{C}^{\infty}-$rings $\{ (V_i,\Phi_i), \nu_{ij}: (V_i,\Phi_i) \to (V_j,\Phi_j)\}_{i,j \in I}$ we have

$${\rm Hom}_{\mathcal{C}^{\infty}{\rm \bf {\rm \bf vNRng}}}(V, \varinjlim_{i \in I} V_i) = \varinjlim_{i \in I} {\rm Hom}_{\mathcal{C}^{\infty}{\rm \bf {\rm \bf vNRng}}}(V, V_i),$$
(cf. \textbf{Proposition 3.8.14} of \cite{Borceux2})
\end{remark}

Consider the category whose objects are all finitely presented von Neumann regular $\mathcal{C}^{\infty}-$ring and whose morphisms are the $\mathcal{C}^{\infty}-$homomorphisms between them, and denote it by $\mathcal{C}^{\infty}{\rm \bf {\rm \bf vNRng}}_{\rm f.p.}$.\\

\begin{remark}We have:
$${\rm Obj}\,(\mathcal{C}^{\infty}{\rm \bf Rng}_{\rm fp})\cap {\rm Obj}\,(\mathcal{C}^{\infty}{\rm \bf {\rm \bf vNRng}}) \subseteq {\rm Obj}\,(\mathcal{C}^{\infty}{\rm \bf {\rm \bf vNRng}}_{\rm fp}).$$

\end{remark}

Thus, keeping in mind the remarks above, and following the same line of the developments made in the  section 2 we obtain:

\begin{theorem}The category $$\mathcal{C}^{\infty}{\rm \bf {\rm \bf vNRng}}_{{\rm fp}}^{{\rm op}}$$ is a category with finite limits freely generated by the
von Neumann regular $\mathcal{C}^{\infty}-$ring  ${\rm vN}(\mathcal{C}^{\infty}(\mathbb{R}))$, i.e.,
for any category with finite limits $\mathcal{C}$, the evaluation of a left-exact functor $F: \mathcal{C}^{\infty}{\rm \bf {\rm \bf vNRng}}_{{\rm fp}}^{{\rm op}} \to \mathcal{C}$ at ${\rm vN}\,(\mathcal{C}^{\infty}(\mathbb{R}))$ yields the following equivalence of categories:

$$\begin{array}{cccc}
    {\rm ev}_{{\rm vN}(\mathcal{C}^{\infty}(\mathbb{R}))}: & {\rm \bf Lex}\,(\mathcal{C}^{\infty}{\rm \bf {\rm \bf vNRng}}_{{\rm fp}}^{{\rm op}}, \mathcal{C}) & \rightarrow & \underline{\mathcal{C}^{\infty}-{\rm {\rm \bf vNRng}}}\,(\mathcal{C}) \\
     & F & \mapsto & F({\rm vN}\,(\mathcal{C}^{\infty}(\mathbb{R})))
  \end{array}$$

\end{theorem}

Combining the results presented in this section and the one stated in the {section 1}  on classifying topoi,  we obtain the  following:

\begin{theorem}The presheaf topos ${\rm \bf Sets}^{\mathcal{C}^{\infty}{\rm \bf {\rm \bf vNRng}}_{{\rm fp}}}$ is a classifying topos for von Neumann regular $\mathcal{C}^{\infty}-$rings, and the universal von Neumann regular $\mathcal{C}^{\infty}-$ring $R$ is the von Neumann regular $\mathcal{C}^{\infty}-$ring object in ${\rm \bf Sets}^{\mathcal{C}^{\infty}{\rm \bf {\rm \bf vNRng}}_{{\rm fp}}}$ that is given by ${\mathcal{C}^{\infty}{\rm \bf vNRng}_{{\rm fp}}}({\rm vN}\,(\mathcal{C}^{\infty}(\mathbb{R}),-))$, thus it is naturally isomorphic to  the  forgetful functor from $\mathcal{C}^{\infty}{\rm \bf {\rm \bf vNRng}}_{{\rm fp}}$ to ${\rm \bf Sets}$.
Thus, for any Grothendieck topos $\mathcal{E}$ there is an equivalence of categories, natural in $\mathcal{E}$:

$$\begin{array}{cccc}
     & {\rm Geom}\,(\mathcal{E},{\rm \bf Sets}^{\mathcal{C}^{\infty}{\rm \bf {\rm \bf vNRng}}_{{\rm fp}}} ) & \rightarrow & \mathcal{C}^{\infty}{\rm {\rm \bf vNRng}}\,(\mathcal{E}) \\
     & f & \mapsto & f^{*}(R)
  \end{array}$$
\end{theorem}

\section{Final remarks and future works}

\hspace{0.5cm}We have described classifying toposes for three theories: the theory of $\mathcal{C}^{\infty}-$rings and the theories of local and of von Neumann regular $\mathcal{C}^{\infty}-$rings.
In \cite{rings2}, I. Moerdijk, N. van Qu\^{e} and G. Reyes present the classifying topos for the (geometric) theory of Archimedean $\mathcal{C}^{\infty}-$rings. This reinforce the following questions:

\begin{itemize}
    \item{Are there other sensible descriptions of classifying toposes for other distinguished classes of $\mathcal{C}^{\infty}-$rings?}
    \item{In particular, is there a nice description of the theory of von Neumann regular $\mathcal{C}^{\infty}-$rings in the language of $\mathcal{C}^{\infty}-$rings, $\mathcal{L}$ (without the need for the new symbol for the ``quasi-inverse'')?}
\end{itemize}

In the paper (under preparation) \cite{BM}, we use von Neumann regular $\mathcal{C}^{\infty}-$rings in order to classify Boolean algebras. We show that a von Neumann regular $\mathcal{C}^{\infty}-$ring is isomorphic to the $\mathcal{C}^{\infty}-$ring of global sections of the structure sheaf of its affine $\mathcal{C}^{\infty}-$scheme (see \cite{Joyce}). Such results motivate us to look for similar characterizations for some distinguished classes of $\mathcal{C}^{\infty}-$rings in terms of its $\mathcal{C}^{\infty}-$spectrum topology. \\





\end{document}